\documentclass[12pt,reqno]{amsart}
\usepackage{amsmath,amsfonts,amsthm,amssymb,color,hyperref,verbatim, dsfont}
\usepackage{float}
\usepackage[T1]{fontenc}
\usepackage{mathtools}
\usepackage{graphicx}
\usepackage{dsfont}
\usepackage{subfigure} 
\usepackage{enumitem}
\makeatletter
     \def\section{\@startsection{section}{1}%
     \z@{.7\linespacing\@plus\linespacing}{.5\linespacing}%
     {\bfseries
     \centering
     }}
     \def\@secnumfont{\bfseries}
     \makeatother

\usepackage{graphicx}

\newcommand{\R}{\mathbb R}

\newcommand{\E}{\mathbb E}
\newcommand{\PP}{\mathbb P}

\newcommand{\ep}{\varepsilon}

\newcommand{\lp}{\left(}
\newcommand{\rp}{\right)}
\newcommand{\lc}{\left[}
\newcommand{\rc}{\right]}
\newcommand{\lcl}{\left\{}
\newcommand{\rcl}{\right\}}

\newcommand{\cexp}{2H} 

\newtheorem{theorem}{Theorem}[section]
\newtheorem{lemma}[theorem]{Lemma}
\newtheorem{prop}[theorem]{Proposition}

\theoremstyle{definition}
\newtheorem{defn}[theorem]{Definition}
\theoremstyle{remark}
\newtheorem{remark}{Remark}

\numberwithin{equation}{section}
\newcommand{\eg}{E_Z(\gamma)}
\newcommand{\ls}{\mathcal{L}_Z(x)}
\newcommand{\wt}[1]{\widetilde{#1}}
\newcommand{\leb}[1]{\text{Leb}(#1)}
\newcommand{\pix}[1]{\text{pix}(#1)}
\newcommand{\dlog}[1]{\text{Den}_{\log}\left( #1 \right)}
\newcommand{\dpix}[1]{\text{Den}_{pix} \left( #1 \right)}
\newcommand{\dhaus}[1]{\text{Dim}_{H} \left( #1 \right)}
\newcommand{\dpack}[1]{\text{dim}_P \left( #1 \right)}
\newcommand{\dpackprof}[2]{\text{dim}_{P, #1} \left( #2 \right)}
\newcommand{\dbox}[1]{\text{dim}_B \left( #1 \right)}
\newcommand{\dtheta}[1]{\text{dim}_{\theta} \left( #1 \right)}
\newcommand{\ds}{\displaystyle}
\newcommand{\chaus}[1]{\text{dim}_{H} \left(#1\right)}
\newcommand{\dlbox}[1]{\text{\underline{dim}}_{B} \left( #1 \right)}
\newcommand{\dubox}[1]{\overline{\text{dim}}_{B} \left( #1 \right)}
\newcommand{\dltheta}[1]{\text{\underline{dim}}_{\theta} \left( #1 \right)}
\newcommand{\dutheta}[1]{\overline{\text{dim}}_{\theta} \left( #1 \right)}
\newcommand{\dlthetaprof}[2]{\text{\underline{dim}}_{\theta, #1}  \left(#2\right)}
\newcommand{\duthetaprof}[2]{\overline{\text{dim}}_{\theta, #1} \left( #2  \right)}

\newcommand{\sap}[1]{\text{s}^\alpha\text{-p}  \left( #1 \right)}
\newcommand{\saP}[1]{\text{s}^\alpha\text{-P}  \left( #1 \right)}
\newcommand{\diam}[1]{\lvert #1\rvert}

\begin{document}
\title[Sojourn times]{Fractal dimensions of the Rosenblatt process}
\author[L. Daw]{Lara Daw}
\address{University of Luxembourg, Department of Mathematics, Luxembourg}
 \email{lara.daw@uni.lu}
\author[G. Kerchev]{George Kerchev}
\address{University of Luxembourg, Department of Mathematics, Luxembourg}
 \email{george.kerchev@uni.lu}

\begin{abstract}
The paper concerns the image, level and sojourn time sets associated with sample paths of the Rosenblatt process. We obtain results regarding the Hausdorff (both classical and macroscopic), packing and intermediate dimensions, and the logarithmic and pixel densities. As a preliminary step we also establish the time inversion property of  the Rosenblatt process, as well as some technical points regarding the distribution of $Z$.
\end{abstract}

\maketitle



\allowdisplaybreaks

\section{Introduction}
The Rosenblatt process $Z = (Z_t)_{t \geq 0}$ is a stochastic process that is a limit of normalized sums of long-range dependent random variables. It belongs to the class of Hermite processes and is the simplest member that is non-Gaussian. It has continuous but nowhere differentiable paths and is selfsimilar of order $H \in (1/2, 1)$ with stationary increments.

The process $Z$, due to its self-similarity, can find applications across a multitude of fields like  internet traffic~\cite{chaurasia2019performance}, hydrology, and turbulence~\cite{sakthivel2018retarded, lakhel2019existence}. We refer the reader to~\cite{embrechts-maejima-2002} and~\cite{samorodnitsky-taqqu-1994} for a detailed review of the properties associated with self-similarity. In particular,  the Rosenblatt process is used in finance~\cite{torres-tudor-2009, stoyanov-etal-2019, fauth-tudor-2016} and statistical inference~\cite{leduc-etal-2011, dehling-rooch-taqqu-2013, nourdin-tran-2019}.

From a mathematical standpoint the process has received a lot of interest since its inception in~\cite{rosenblatt-1956}. Its distribution is not known in explicit form but was studied first in~\cite{albin-1998} and more recently in~\cite{maejima-tudor-2013} and~\cite{veillette-taqqu-2013}. There are three integral representations: in terms of time, the spectrum and on finite intervals, see~\cite{taqqu-2011}. There is also a wavelet representation~\cite{pipiras-2004} (see also the recent article~\cite{ayache-esmili-2020} for the wavelet representation of the generalized Rosenblatt process and its rate of convergence). From a statistical point of view, the value of the Hurst index $H$ is important for practical applications and various estimators exist, see~\cite{bardet-tudor-2010, tudor-viens-2009}.

In the present paper, we focus on the fractal properties of the random sets and measures determined by the sample paths of $Z$, i.e.,  if the underlying probability space is $(\Omega, \mathcal{F}, \PP)$, we study the function $Z(t) = Z_t( \omega)$, for a fixed $\omega \in \Omega$. Some (random) sets of interest are then:
\begin{align}
	& \label{defImgS}\text{Image set: } Z(E) \coloneqq \lcl Z(t) : t \in E \rcl; \\
	& \label{defGS} \text{Graph set: } Gr_Z(E) \coloneqq \lcl (t, Z(t) ) \in E \times \R : t \in E \rcl; \\
	&\label{defLS} \text{Level set: } \ls \coloneqq \lcl t \in \R_+ : Z(t) = x \rcl, x \in \R; \\
	& \label{defSS}\text{Sojourn set: } \eg \coloneqq \lcl t \in \R_+ : |Z(t) | \leq t^\gamma \rcl,  \gamma > 0;\\
	& \label{defInvS}\text{Inverse image: } Z^{-1} (E') \coloneqq \lcl t \in \R_+ : Z(t) \in E' \rcl,
\end{align}
where $E \subset \R_+$ and $E' \subset \R$ are  Borel sets. These sets , due to self-similarity property of $Z$, may look like a fractal, see, e.g., Figure~\ref{fig:sojourn}, for the sojourn set of the Rosenblatt process. In order to describe such sets quantitatively one can use a type of fractal dimension. 

\begin{figure}[H]
	\includegraphics[width=10cm]{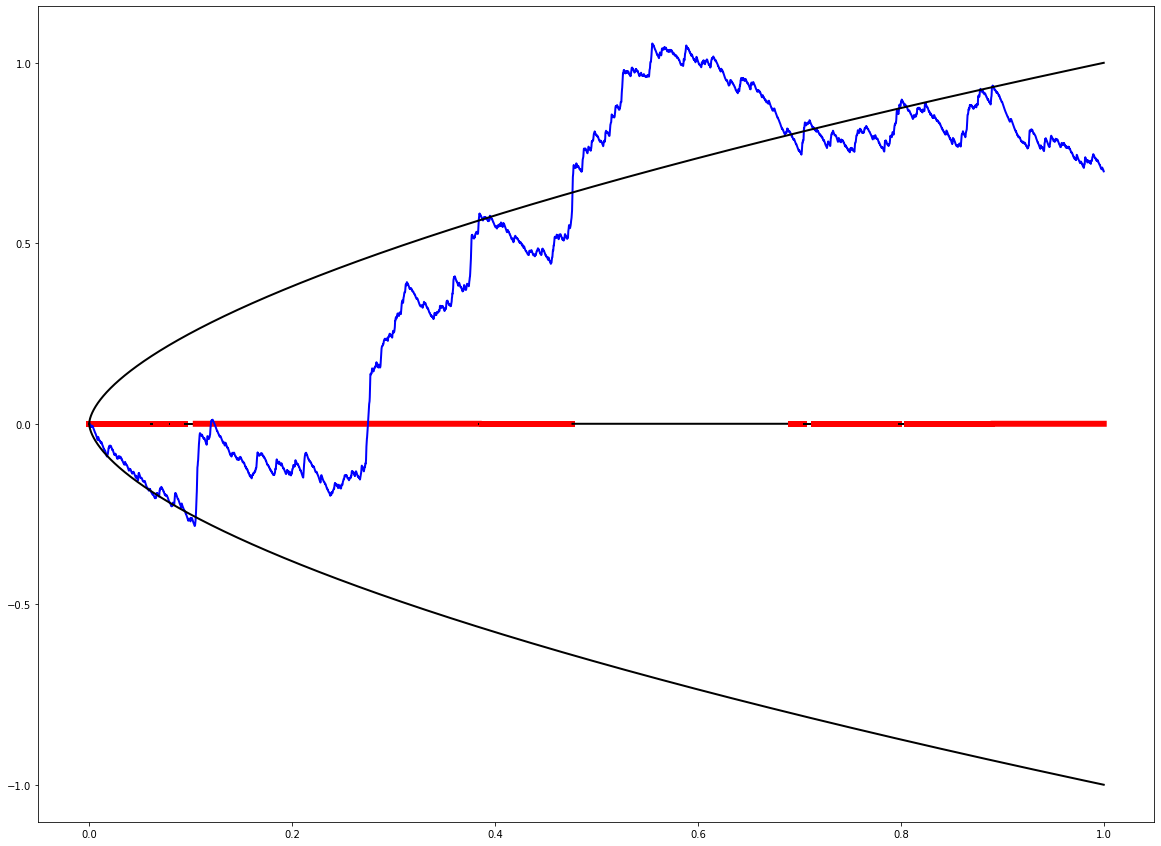}
	\caption{Simulation of a  Rosenblatt process of Hurst index $H = 0.6$. In red - the sojourn set $\eg$ for $\gamma = 0.6$. }
	\label{fig:sojourn}
\end{figure}

Fractal dimensions give you an intuition about the geometry of a set. Having identified some interesting random sets and possible ways to measure them, we note that such studies can be traced to the pioneering work of L\'evy~\cite{levy-1953} and Taylor~\cite{taylor-1953, taylor-1955, taylor-1967} on the sample path properties of the Brownian motion. We refer the reader to~\cite{sato-2013} and~\cite{xiao-2004} for surveys of such results for L\'evy and Markov processes respectively.

An important class of such dimensions reflects local properties of the set. One important example is the \emph{classical Hausdorff dimension}, which can be defined as follows using the Hausdorff content, see~\cite[Section 3.2]{falconer2004fractal}. For $E \subset \R$,

\begin{align}
	\label{eq:classical_hausdorff} \chaus E \coloneqq \inf \lcl s \geq 0 \, : \, \forall \epsilon>0,  \exists \mbox{  cover} \left\{ U_i\right\}_{i = 1}^\infty \mbox{ of } E, \mbox{ s.t. } \sum_{i =1}^\infty |U_i|^s \leq \epsilon \rcl,
\end{align}
where $|F|$ denotes the {\it diameter} of the set $F$. Moreover by imposing further restrictions on the sets in the cover $\lcl U_i \rcl$ one can recover the definitions of \emph{box dimension}. In particular, for $E \subset \R$, the \emph{lower box dimension} is given by: 
\begin{align}
\label{eq:lower-box-dim}  \dlbox{E} \coloneqq \inf  \lcl s \geq 0 : \begin{array}{c} \forall \epsilon>0,  \exists \mbox{  cover } \left\{ U_i\right\}_{i = 1}^\infty \mbox{ of } E, \mbox{ s.t. } \\ \lvert U_i \rvert = \lvert U_j \rvert~ \forall i,j \mbox{ and } \sum_{i =1}^\infty |U_i|^s \leq \epsilon \end{array} \rcl.
\end{align}
Similarly, we define the \emph{upper box dimension}: 
\begin{align}
\label{eq:upper-box-dim}  \dubox{E} \coloneqq \inf  \lcl s \geq 0 : \begin{array}{c} \forall \epsilon>0,  \exists \delta > 0, \forall \mbox{  cover } \left\{ U_i\right\}_{i = 1}^\infty \mbox{ of } E, \mbox{ s t. } \\ \lvert U_i \rvert \leq \delta, \, \lvert U_i \rvert = \lvert U_j \rvert~  \forall i,j \mbox{ and } \sum_{i =1}^\infty |U_i|^s \leq \epsilon \end{array} \rcl.
\end{align}
The box dimension $\dbox{E}$ is then given by the common value (if it exists) of  $\dlbox{E}$ and $\dubox{E}$. Next for $\theta \in [0, 1]$, the \emph{$\theta$-intermediate dimensions} $\dtheta{E}$   is a dimension that interpolate between the \emph{Hausdorff} and \emph{box dimensions} by increasing restriction on the relative sizes of covering sets as $\theta$ increases ($\delta^{1/\theta} \leq |U_i| \leq \delta$ for all $i$). In particular, one defines $\dltheta{E}$ and $\dutheta{E}$ similarly to $\dlbox{E}$ and $\dubox{E}$. Then $\dtheta{E}$ is the common value if it exists of $\dltheta{E}$ and $\dutheta{E}$.

One need not consider only covers for the set $E$. For example,  $\dbox{E}$ can be defined alternatively using coverings by small balls of equal radius (corresponding to $\dlbox{E}$) or using packings by disjoint balls of equal radius that are as dense as possible (corresponding to $\dubox{E}$), see~\cite[Section 3.4]{falconer2004fractal}. If the radii are allowed to differ the covering procedure corresponds to the classical Hausdorff dimension while the packing one is associated to the \emph{packing dimension} $\dpack{E}$. In linear programming the packing and covering problems are dual of each other and thus the packing dimension can be considered as the \emph{dual analogue} to the classical Hausdorff dimension. The precise definitions are delayed to Section~\ref{sec:image_set}. 

Other definitions of the packing and intermediate dimensions are possible by employing methods from potential theory. Thus, $\dpack{E}$, $\dltheta{E}$ and $\dutheta{E}$ can be expressed via capacities with respect to certain kernels, see~\cite{falconer1997packing, burrell2019projection}. This gives rise to \emph{packing and intermediate dimension profiles} - $\dpackprof{\alpha}{E}$, $\dlthetaprof{\alpha}{E}$ and $\duthetaprof{\alpha}{E}$  respectively. See ~\ref{sec:inter_dim} for the precise definitions.

All the dimensions in the discussion above pertain to local properties of the set. It is often the case, for instance in statistical physics, that one needs to quantify global properties of an infinite set. The simplest way of assessing the size of such a set is given by its (Lebesgue) density at infinity. In particular, we utilize the \emph{logarithmic density} $\dlog{E}$ and the \emph{pixel density} $\dpix{E}$ (the latter corresponding to the ``pixelated" image). Alternatively, one can use the \emph{macroscopic Hausdorff dimension} $\dhaus{E}$ introduced in~\cite{barlow-taylor-1989, barlow-taylor-1992} for  the study of the macroscopic properties of random walks. More recent applications can be found in the study of high peaks of solutions of the stochastic heat equation~\cite{khoshnevisan-kim-xiao-2017, khoshnevisan-xiao-2017}. Definitions of these concepts are provided in Section~\ref{sec:level_set}. A brief summary of all dimensions discussed can be seen in Table~\ref{table:dim}.

\begin{table}[h]
	\begin{center}
		\begin{tabular}{ |c|c|c|c |c | c|  }
			\hline 
			Dimension  & Name  & Cover & Size & Values & Limit \\
			\hline
			$\chaus{\cdot}$ &   \begin{tabular}{@{}c@{}} Classical \\ Hausdorff \end{tabular} & Covering & $ (0,  \delta]$ & $[0,1]$ & $\delta \to 0$ \\
			\hline
			$\dbox{\cdot}$ & Box & \begin{tabular}{@{}c@{}} Upper - Covering\\ Lower - Packing \end{tabular} & $\delta$ & $[0, 1]$ & $\delta \to 0$ \\ 
			\hline 
			$\dpack{\cdot}$ &   Packing & Packing & $(0,  \delta]$ & $[0,1]$ & $\delta \to 0$ \\
			\hline
			$\dtheta{\cdot} $ & Intermediate & & $\in (\delta^{1/\theta}, \delta)$  & $[0, 1]$ & $\delta \to 0$ \\ 
			\hline 
			\hline
			$\dlog{\cdot}$ & \begin{tabular}{@{}c@{}} Logarithmic\\ density \end{tabular}  & Interval & $[1, 2^n]$ & $[0,1]$ & $n \to \infty$ \\
			\hline
			$\dpix{\cdot} $ & \begin{tabular}{@{}c@{}} Pixel\\ density \end{tabular}& Interval* & $[1, 2^n]$  & $[0, 1]$ & $n \to \infty$ \\ 
			\hline
			$\dhaus{\cdot}$ &   \begin{tabular}{@{}c@{}} Macroscopic\\ Hausdorff \end{tabular} & \begin{tabular}{@{}c@{}} Collections of sets\\ in $[2^{n-1}, 2^n) $ \end{tabular} & $(0,  \delta]$ & $[0,1]$ & $n \to \infty$ \\
			\hline		
		\end{tabular}
		\medskip
		\caption{Overview of the types of fractal dimensions. For the pixel density the cover consists of the integer points in the interval at distance less than $1$ from $E$. }
		\label{table:dim}
	\end{center}
\end{table}
We also mention a few relations between the dimensions mentioned so far to give the reader some intuition:
\begin{align}
	\notag \chaus{E} & \leq \dlbox{E} \leq \dubox{E}; \quad \chaus{E}  \leq \dltheta{E} \leq \dutheta{E} \leq \dubox{E};  \\
	\notag \dpack{E} & \leq  \dubox{E}; \quad \dlog{E}  \leq \dpix{E}.
\end{align}

Before we list our main results, we outline what is known regarding fractional properties of sample paths of a Hermite process of rank $1$, i.e., the fractional Brownian motion. The fractional Brownian motion $X = (X_t)_{t \geq 0}$, like $Z$, is a selfsimilar stochastic process with stationary increments. Both processes, $X$ and $Z$, share the same covariance structure and are governed by a parameter $H$ (called Hurst parameter in both cases). Unlike the Rosenblatt process, the process $X$ is Gaussian and $H \in (0,1)$. See Table~\ref{table:fbm} for an overview of some fractal properties of sets associated with the sample paths of the fractional Brownian motion. 

\begin{table}[h]
	\begin{center}
		\begin{tabular}{ |c|c|c|c | }
			\hline 
			& $X(E)$  & $\mathcal{L}_X(x)$  & $E_X(\gamma)$   \\
			\hline
			$\dpack \cdot$ & $\frac{1}{H} \dpackprof{H}{E}$~\cite{xiao-1997}&  &1 \\ 
			\hline 
			$\dltheta \cdot $ &  $\frac{1}{H} \dlthetaprof{H}{E}$~\cite{burrell2020dimensions}   & &1\\
			\hline 
			$\chaus \cdot $ & \begin{tabular}{@{}c@{}} $\min \left(1, \frac{1}{H} \chaus{E}\right)$ \\ \cite{kahane-1985} \end{tabular} & $1 - H$~\cite{falconer2004fractal} &1 \\
			\hline 
			$\dhaus \cdot$  & &    $1 - H$~\cite{daw-2020}& $1 - H$~\cite{Nourdin-Peccati-Seuret-2018} \\ 
			\hline 
			$\dpix \cdot$ & & &  $\gamma + 1 - H$~\cite{Nourdin-Peccati-Seuret-2018} \\ 
			\hline 
			$\dlog \cdot$ & & &  $\gamma + 1 - H$~\cite{Nourdin-Peccati-Seuret-2018}  \\
			\hline 
			
		\end{tabular}
		\medskip
		\caption{Table of fractal dimensions and densities of random sets associated with the fractional Brownian motion with $\gamma \in [0, H)$.}
		\label{table:fbm}
	\end{center}
\end{table}

For completeness  we mention also some results regarding the graph and the inverse sets. If $X : \R^N \mapsto \R^d$ is a fractional Brownian sheet, it has been proved in~\cite{adler-1977} that, almost surely,  $\chaus{Gr_X([0,1]^N)}= \min \left\{N/H, N + (1 - H)d\right\}$. The  box dimension of the graph of the fractional Brownian sheet over a non degenerate cube $Q$ of $\R^N$ was determined in~\cite{kamont1995fractional}. Moreover, with probability 1, $\dbox{Gr_X\left(Q\right)}=N+1 - H$. Regarding the inverse set, the following holds: for $E $  a closed subset of $\R^d$, $\chaus{X^{-1}\lp E\rp}= N-Hd +\chaus{E}$ (see~\cite{monrad-pitt-1987}). We believe that analogous results can be established for the Rosenblatt process, but the sets in question are not the subject of the current paper.

Many of the results listed above rely on H\"older regularity conditions for the sample paths, and more precisely, for the local time of the process. Such properties have been established for stationary Gaussian processes, like the fractional Brownian motion, by Berman in~\cite{berman-1969}. His analytic approach, which is based on properties of the Fourier transform of the underlying process, has been adapted to the Rosenblatt setting in~\cite{shevchenko-2010} where existence of the local time of $Z$ was first established. H\"older regularity was then recovered in the recent paper~\cite{knsv-2020}. These new results now allow  to generalize some of the results in Table~\ref{table:fbm} for the Rosenblatt case.  See Table~\ref{table:rp}.

\begin{table}[h]
	\begin{center}
		\begin{tabular}{ |c|c|c|c | }
			\hline 
			& $Z(E)$  & $\ls$  & $\eg$  \\
			\hline
			$\dpack \cdot$ & $\frac{1}{H} \dpackprof{H}{E}$~\cite{shieh-xiao-2010}& $ 1-H$ &1  \\ 
			\hline 
			$\dltheta \cdot $ &  $\frac{1}{H} \dlthetaprof{H}{E}$ &$1-H$  & 1\\
			\hline 
			$\chaus \cdot $ & \begin{tabular}{@{}c@{}} $ \min \left(1, \frac{1}{H} \chaus{E}\right)$ \\ \cite{shieh-xiao-2010} \end{tabular} & $1-H$ &1  \\
			\hline 
			$\dhaus \cdot$  & &    $1 - H$ & $1 - H$   \\ 
			\hline 
			$\dpix \cdot$ & & &  $\gamma + 1 - H$ \\ 
			\hline 
			$\dlog \cdot$ & & &  $\gamma + 1 - H$  \\
			\hline 
			
		\end{tabular}
		\medskip
		\caption{Table of fractal dimensions and densities of random sets associated with the Rosenblatt process with $\gamma \in [0, H)$. }
		\label{table:rp}
	\end{center}
\end{table}

All results in Table~\ref{table:rp} but the ones for the dimensions of the image of the process $Z(E)$ are new. Our findings are collected in the following three propositions. First, for the image set we extend the results of~\cite{shieh-xiao-2010} to the intermediate dimensions setting, as in~\cite{burrell2020dimensions}:

\begin{theorem}\label{thm:images}
	Let  $\theta \in (0,1]$ and $E \subset \R^+$ be compact. Then almost surely:
	\begin{align}
		\dltheta {Z(E)}= \frac{1}{H} \dlthetaprof{H}{E},
	\end{align}
	and 
	\begin{align}
		\dutheta {Z(E)}= \frac{1}{H} \duthetaprof{H}{E},
	\end{align}
	where $\dlthetaprof{H}{\cdot}$ and $\dlthetaprof{H}{\cdot}$ are the lower and upper $\theta$-intermediate dimension profiles respectively. For the precise techinical definitions of these two objects see~\eqref{eq:dltp} and~\eqref{eq:dutp} in Section~\ref{sec:inter_dim}.  
\end{theorem}

Then, we study the proportion of time spent by a stochastic process in a given region.  We describe the size of  the level sets $\ls$ in terms of intermediate dimensions and macroscopic Hausdorff dimension. The following holds:

\begin{theorem}\label{thm:levels} For $E \subset \R$ and $\theta \in [0,1]$, let   $\dtheta{E}$ and $\dhaus{E}$ denote the $\theta$-intermediate and macroscopic Hausdorff dimensions of $E$.  Then, for any $x \in \R$ and $0 < \ep < 1$,
	\begin{align}\label{eq:levelset_classical}
		\forall x \in \R, \mathbb{P} \left(\dtheta {\ls \cap [\ep, 1]} = 1 - H \right)  = 1.
	\end{align}
	And,
	\begin{align} \label{eq:levelset_packing}
		\forall x \in \R, \mathbb{P}(\dpack{\ls \cap [\ep, 1]} = 1-H)=1.
	\end{align}
	Moreover, 
	\begin{align} \label{eq:levelset_as}
		\mathbb{P}(\forall x\in\mathbb{R}:\,\dhaus{\ls} = 1-H)=1.
	\end{align}

\end{theorem}
We believe that the same uniform result holds for classical Hausdorff dimension but we only prove the pointwise one. Finally, we establish the results for the sojourn times $\eg$:

\begin{theorem}\label{thm:sojourn}  For $E \subset \R$, let   $\dpix{E}$ and $\dlog{E}$ denote the pixel and logarithmic densities of $E$. Then, for all $\gamma \in [0,H]$,  
	\begin{align}
		\label{eq:densities}\dpix{\eg}= \dlog{\eg}= \gamma + 1-H, \quad \mbox{a.s.}
	\end{align} 
	Moreover, 
	\begin{align} \label{eq:haus_as}
		\dhaus \eg = 1-H \quad \mbox{a.s.}
	\end{align}
\end{theorem}

To fill the missing entries in Table~\ref{table:rp} one needs new techniques. In particular, the macroscopic Hausdorff dimension and the two densities of the image set $Z(E)$ should depend on the fractional properties of $E$ (in particular should be $0$ if $E$ is bounded). However, intuition regarding this relation is missing. Regarding, the level set $\mathcal{L}_Z(x)$, the approach for the macroscopic Hausdorff dimension does not translate since the key result (Lemma~\ref{LemmaHD}) is an artifact of the definition of $\text{Dim}_H$.   

The authors believe that many of the results above can be extended to some generalizations of the Rosenblatt process, for instance, when the time and space sets are $N$ and $d$ dimensional, or when the Hurst index is a function of time, as in~\cite{shevchenko-2010}. To ease the presentation only the case $N = d = 1$ and $H \in (1/2, 1)$ - fixed is considered.  However, establishing the results for Hermite processes of rank above $2$ requires new techniques and is beyond the scope of the current paper. In particular, the Berman analytic approach relies on a ``good'' representation for the Fourier transform of the process and this is not known for Hermite processes of higher rank.

The structure of the paper is as follows. The three main results listed above are established in Sections~\ref{sec:image_set}-\ref{sec:sojourn_set}. Some necessary technical  properties of the Rosenblatt process are reviewed and proved in Section~\ref{sec:rosenblatt}.

\section{Properties of the Rosenblatt process}\label{sec:rosenblatt}

We start this section with some basic properties of the stochastic process $Z$. As mentioned in the introduction, there exist a few integral representations of $Z$. For our purposes, of special interest is the spectral representation (see~\cite{taqqu-2011} and~\cite{Dobrushin-Major-1979}):
\begin{align}
	\label{eq:rosen_spectral} Z_t^H = C(H) \int_{\R^2} \frac{e^{i (x + y)t } - 1}{i (x + y)}  Z_G(dx) Z_G(dy),
\end{align} 
where the double Wiener-Ito integral is taken over $x \neq \pm y$ and $Z_G(dx)$ is a complex-valued random white noise with control measure $G$ satisfying $G(tA) = t^{1 - H} G(A)$ for all $t \in \R$ and $G(dx) = |x|^{-H} dx$. The constant $C(H)$ in~\eqref{eq:rosen_spectral} is such that  
\begin{align}
	\notag \E[ Z_t^2 ] = t^{2H} \quad \text{ and } \quad \E[ Z_t Z_s ] = \frac{1}{2} \lp t^{2H} + s^{2H} - |t - s|^{2H} \rp,
\end{align}
for all $s, t \geq 0$.

\begin{remark}
	Note that in the notation of~\cite{taqqu-2011}, $Z_G(dx) = |x|^{-H/2} d\hat{B}(x)$, with $(B(t))_{t \in \R}$ the Brownian motion and $d\hat{B}(x)$ is viewed as the complex-valued Fourier transform of $dB(x)$. For more details, see~\cite{Taqqu-1979}.
\end{remark}

It is known (see \cite{tudor2008analysis}) that the Rosenblatt process has the following properties: 
\begin{enumerate}
	\item[\textbf{(1)}] \textbf{self-similarity: } $Z$ is $H$-self-similar; that is, the processes $ \left\{Z_{ct}, \, t \geq 0 \right\}$ and $\left\{c^HZ_{t}, \, t \geq 0 \right\}$ have the same distribution.
	\item[\textbf{(2)}] \textbf{stationary increments: } $Z$ has stationary increments; that is, the distribution of the process $\left\{Z_{t+s}- Z_s, \, t \geq 0\right\}$ does not depend on $s \geq 0$.
	\item[\textbf{(3)}] \textbf{continuity:} the trajectories of the Rosenblatt process $Z$ are $\delta$-H\"older continuous for every $\delta< H$. 
\end{enumerate}
We will mention one more property that will be needed in our proofs, and is a consequence of the finite time interval representation~\cite[Section 7.3]{taqqu-2011} of the Rosenblatt process. The natural filtration associated to a Rosenblatt process is Brownian, i.e., there is  a Brownian motion $(B_t)_{t\geq 0}$ defined on the same probability space than $Z$ such that its filtration satisfies 
\begin{align}
	\label{filtrationB}
	\sigma\left\{Z_s \, : \, s \leq t\right\} \subset \sigma\left\{B_s \, : \, s \leq t\right\},
\end{align}
for all $t>0$.
\\ Moreover, by~\cite[Theorem 1.1]{maejima-tudor-2013}, for any $d \geq 1$ and $t_1, \ldots, t_d \geq 0$,
\begin{align}
	\label{eq:repn_thorin} (Z_{t_1}, \ldots , Z_{t_d} ) \overset{d}{=} \lp \sum_{n=1}^\infty \lambda_n(t_1) (\ep_n^2 - 1), \ldots, \sum_{n=1}^\infty \lambda_n(t_d) (\ep_n^2 - 1) \rp,
\end{align}
where $( \ep_n)_{n \geq 1}$ are i.i.d $\mathcal{N}(0,1)$ random variables and $(\lambda_n(t))_{n \geq 1}$ are the (real) eigenvalues of a self-adjoint Hilbert-Schmidt operator associated with the process $Z$ (see~\cite{Dobrushin-Major-1979}).

For our analysis a few properties of the density for the joint process $(Z_{t_1}, Z_{t_2})$ are needed. Using techniques from~\cite{knsv-2020} we can establish the following:

\begin{prop}\label{prop:rosen_density} \quad
	\begin{enumerate}[label=(\roman*)]
		\item The probability density function $f : \mathbb{R} \to \mathbb{R}_+$ of $Z_1$ is  continuous and $f(x) > 0$ for $x \geq 0$.    \label{prop:unimodal} 
		\item For every $t_1, \ldots, t_n \geq 0$, the vector $(Z_{t_1}, \ldots, Z_{t_n})$ has a continuous density.   \label{prop:cont_density}
	\end{enumerate}
	
\end{prop}

\begin{proof}[Proof of Proposition~\ref{prop:rosen_density}]
	(i)  The density $f$ of $Z_1$ is continuous (see~\cite[Corollary 4.3]{veillette-taqqu-2013}) and unimodal (see~\cite{maejima-tudor-2013}). Therefore $f(0) > 0$ since $\E[Z_1] = 0$. To see that $f(y) > 0$ for  all $y > 0$, recall~\cite[Corrolary 4.5]{veillette-taqqu-2013}: for $\alpha > 0$, 
	\begin{align}
		\notag \lim_{u \to \infty} \frac{\PP(Z_1 > u + \alpha)}{\PP(Z_1 > u)} = c_H,
	\end{align}
	for a deterministic constant $c_H > 0$. In particular, this shows that for every $y \in \R_+$, there is $x > y$, such that $f(x) > 0$. Combined with the fact that $f$ is continuous and unimodal, this implies that $f(x) > 0$ for every $x \in \R_+$.
	
	(ii) If the characteristic function $\hat{\mu}(z)$ of a probability measure $\mu$ in $\R^d$ is integrable, then $\mu$ has a continuous density $g(x)$ that tends to $0$ as $|x| \to \infty$ (see~\cite[Proposition 2.5(xii)]{sato-2013}). Therefore, it is enough to show that for all $t \in \R_+^n$ and  $\xi \in \R^d$:
	\begin{align}
		\notag \int_{\R^d} \left| \mathbb{E} \exp \left( i \sum_{j = 1}^n \xi_j Z_{t_j} \right) \right| d \xi < \infty.
	\end{align}
	
	At this point we recall~\cite[Lemma 2.1, 2.2]{knsv-2020}.
	
	\begin{lemma}\label{lem:part_eval} Let $L^2_G(\R)$ be a weighted space with norm $\left\| f\right\|_{L_G^2}^2 := \int_\R \lvert f(x) \rvert^2 G(x)dx $. For $t \in \R_+^n$, $\xi \in \mathbb{R}^n$, let $A_{t, \xi} : L^2_G(\mathbb{R}) \to L^2_G(\mathbb{R})$ be the operator given by 
		\begin{align}
			\notag (A_{t, \xi} f)(x) = \int_{\mathbb{R}}  \sum_{j = 1}^n  \xi_j \frac{e^{i t_j (x - y) } -1 }{i (x - y)} f(y) |y|^{-H/2}dy.
		\end{align}
		Let $(\lambda_k(t, \xi))_{k \geq 1}$ be the set of eigenvalues of $A_{t, \xi}$. Then, 
		\begin{align}
			\notag \left| \mathbb{E} \exp \left( i \sum_{j = 1}^n \xi_j Z_{t_j} \right) \right|  = \prod_{k  \geq 1} \frac{1}{( 1 + 4 \lambda_k(t, \xi))^{1/4}}.
		\end{align}
		Moreover, if $t_0 = 0 < t_1 < \cdots < t_n \leq 1$, for every $k \geq 1$,
		\begin{align}
			\label{eq:bound_svalue} \lambda_k(t, \xi)  \geq C(H) (\max_{1 \leq j \leq n} |\xi_j - \xi_{j-1} | | t_j - t_{j-1}|^H )^2 \tilde{\lambda}_k^4, 
		\end{align}
		\noindent where $\tilde{\lambda}_k \sim k^{- H/2}$ (independent of $t$ and $\xi$), $\xi_0=0$ and $C(H) > 0$ is a constant that only depends on $H$. 
	\end{lemma}

	Now, we follow a similar procedure to the one employed for the proof of~\cite[Proposition 1.3]{knsv-2020}. Let $f_0: \R_+^n \times \R^n \to \R_+$ be given by
	\begin{align}
		\label{eq:f_0def} f_0(t, y) \coloneqq t_1^H |y_1| \vee t_2^H |y_2|  \vee \cdots \vee t_n^H |y_n|.
	\end{align}
	
	Further, let $\xi' = (\xi_1 - \xi_{0}, \xi_2 - \xi_{1}, \dots , \xi_n - \xi_{n-1})$ and $t' = (t_1 - t_{0}, t_2 - t_{1}, \dots , t_n - t_{n-1})$. Then
	\begin{align}
		\notag  & \int_{\R^d} \left| \mathbb{E} \exp \left( i \sum_{j = 1}^n \xi_j Z_{t_j} \right) \right| d \xi  \\
		\notag = & \int_{\R^d} \prod_{k \geq 1} (1 + 4 \lambda_k(\xi, t))^{-1/4} d \xi \\
		\notag \leq & \int_{\R^d} \prod_{k \geq 1} \left( 1 + 4C(H)(\max_{1 \leq j \leq n} |\xi_j - \xi_{j-1} | | t_j - t_{j-1}|^H )^2 \tilde{\lambda}_k^4\right)^{-1/4} d\xi \\
		\label{eq:fourier_1} = & \int_{\R^d} \prod_{k \geq 1} \left(1 + 4 C(H) f_0^2 (t', \xi') \tilde{\lambda}_k^4 \right)^{-1/4} d\xi'.
	\end{align}
	
	Let
	\begin{align}
		\notag  G(s) \coloneqq \prod_{k \geq 1} ( 1 + 4s^2 \tilde{\lambda}_k^4)^{-1/4}.
	\end{align}
	
	We can now switch to polar coordinates in~\eqref{eq:fourier_1} via  $|\xi'|=r'$, $\xi'/r'=w'$:
	\begin{align}
		\notag &  \int_{\R^d} \prod_{k \geq 1} \left(1 + 4 C(H) f_0^2 (t', \xi') \tilde{\lambda}_k^4 \right)^{-1/4} d\xi' \\
		\notag \leq  & C \int_{|w' | = 1} \int_0^\infty (r')^{n-1} G( \sqrt{C(H)} r' f_0(t', w') ) dr' \mathcal{H}^{n-1} (dw') \\
		\notag = & C \lp  \int_0^\infty R^{n-1} G(R) dR  \rp \lp \int_{|w'| = 1} ( f_0(t', w') )^{-n} \mathcal{H}^{n-1} (dw') \rp,
	\end{align}
	where $\mathcal{H}^{n-1} (dw')$ is the $(n-1)$-dimensional Hausdorff measure, $C > 0$ is a constant that depends on $H$ and the last equality follows with the change of variables $R = \sqrt{C(H) }r' f(t', w')$.
	
	Next, recall~\cite[Lemma 2.3]{knsv-2020} that $G(s)$ is finite and positive for any $s > 0$ and moreover there are  constants $c_1, c_2 > 0$ such that for all $\beta \geq 1$, 
	\begin{align}
		\notag \int_0^\infty  s^{\beta - 1} G(s) ds \leq c_2 H  c_1^{- \beta H} \Gamma ( \beta H),
	\end{align}	
	where $\Gamma$ is the Gamma function.
	
	Finally, since $t_1, \ldots, t_n  > 0$ are fixed, by the definition~\eqref{eq:f_0def} of $f_0(t', w')$, 
	\begin{align}
		\notag \int_{|w'| = 1} ( f_0(t', w') )^{-n} \mathcal{H}^{n-1} (dw')  \leq & C(t')  \int_{|w'| = 1} (|w_1| \vee \cdots \vee |w_n|)^{-n} \mathcal{H}^{n-1} (dw') \\
		\notag \leq & C(t') \int_{|w'| = 1} (n^{-1/2})^{-n} \mathcal{H}^{n-1} (dw') < \infty,
	\end{align}
	where $C(t') \coloneqq \inf \{ t_1'^{Hn}, \ldots, t_n'^{Hn} \}$ is a positive constant.

	Therefore, the characteristic function of $(Z_{t_1}, \ldots, Z_{t_n})$ is integrable and thus the joint distribution has a continuous density. 
	
\end{proof}

Next we establish a time inversion property for the Rosenblatt process:
\begin{prop}\label{prop:inverse_time} The inverse time process
	\begin{align} t \mapsto \tilde{Z}_t \coloneqq t^{2H} Z_{1/t},
	\end{align}
	is also a Rosenblatt process. 
\end{prop}

\begin{proof} First, using the spectral representation of a Rosenblatt process~\eqref{eq:rosen_spectral},
	\begin{align}
		\notag t^{2H} Z_{1/t} \overset{d}{=} & \quad C(H) t^{2H} \int_{\R^2} \frac{e^{i (x + y)/t } - 1}{i (x + y)}  Z_G(dx) Z_G(dy) \\
\notag = & \quad C(H) t^{2H}  \int_{\R^2} \frac{e^{i (x' + y')t } - 1}{i (x' + y')t^2}  Z_G(t^2 dx') Z_G(t^2 dy' ),
	\end{align}
with the change of variables $x = x't^2$ and $y = y't^2$.  Now recall  the change of variables formula for the It\^o integral~\cite[Proposition 4.2]{dobrushin-1979}:
	
	\begin{prop}\label{prop:changevar} Let $G$ and $G'$ be two non-atomic spectral measures such that $G$ is absolutely continuous with respect to $G'$, and let $g(x)$ be a complex valued function such that 
		\begin{align}
			\notag g(x) = &  \overline{g(-x)}, \\
			\notag |g^2(x)| = & \frac{d(G(x))}{d(G'(x))}. 
		\end{align}
		Let  $f : \R^2 \to \mathbb{C}$ be a measurable function such that:
		\begin{enumerate}
			\item $f(-x_1, -x_2) = \overline{f(x_1, x_2)}$, and
			\item $||f||^2 = \int |f(x_1, x_2)|^2 G(dx_1) G(dx_n) < \infty$. 
		\end{enumerate}
		
		Then, for $f'(x_1, x_2) = f(x_1, x_2) g(x_1) g(x_2)$, 
		\begin{align}
			\notag  \int f(x_1, x_2) Z_G(dx_1) Z_G(dx_2) \overset{d}{=} \int f'(x_1, x_2) Z_{G'} (dx_1) Z_{G'}(dx_2).
		\end{align}
	\end{prop}
	\medskip 

Let $G_{t^2} (A) \coloneqq G(At^2) = t^{2(1- H) } G(A)$ for every measurable $A$. We apply Proposition~\ref{prop:changevar} with $G$ and $G_{t^2}$,  i.e., with $|g(x)|^2 = t^{2(1- H) }$  a constant depending on $t$. Then, 
	\begin{align}
\notag  &  C(H) t^{2H}  \int_{\R^2} \frac{e^{i (x' + y')t } - 1}{i (x' + y')t^2}  Z_G(t^2 dx') Z_G(t^2 dy' ) \\
\notag = & \quad C(H) t^{2H}  \int_{\R^2} \frac{e^{i (x' + y')t } - 1}{i (x' + y')t^2}  Z_{G_{t^2}}( dx') Z_{G_{t^2}} (dy' ) \\
	\notag  \overset{d}{=} & \quad C(H) t^{2H}  \int_{\R^2} \frac{e^{i (x' + y')t } - 1}{i (x' + y')t^2}  t^{2 (1 - H)} Z_G( dx') Z_G (dy' ) \\
\notag = & \quad C(H)   \int_{\R^2} \frac{e^{i (x' + y')t } - 1}{i (x' + y')} Z_G( dx') Z_G (dy' ),
	\end{align} 
	and we recover the spectral representation of $Z_t$ as desired.

\end{proof}

\begin{remark} For the fractional Brownian motion $B_t^H$ of Hurst index $H \in (1/2, 1)$, the same fact is established using that the process is Gaussian and by comparing covariance functions. However, this property can also be recovered using the approach above. Indeed, we have the following spectral representation:
	\begin{align}
		\notag B_t^H \overset{d}{=} C(H) \int_\R \frac{e^{i \lambda t} - 1}{i \lambda} \frac{1}{|\lambda|^{H - 1/2} }  d \hat{B}(\lambda). 
	\end{align}
	The same change of variables yields the desired conclusion. 
\end{remark}

We also recall a result~\cite[Proposition 4.2]{knsv-2020} regarding oscillations:
\begin{prop}\label{prop:sup-tail} Let $(Z_t)_{t \geq 0}$ be the Rosenblatt process. Then for any $s > 0$ and $h \in (0, s)$, 
	\begin{align}
		\notag \PP\left(\sup_{t\in [s - h, s+ h] }|Z_t-Z_s|\geq u\right) \leq C\exp\left(-\frac{u}{c_1h^{H}}\right),
	\end{align}
	where $c_1$ and $C$ are constants that depend only on $H$. 
\end{prop}

We need the following properties of the local time of the Rosenblatt process. Its existence was shown in~\cite{shevchenko-2010} and one has the representation: 
\begin{align}
	\label{eq:local_rep} L(x, t) = \frac{1}{2 \pi} \int_\R \int_0^t e^{i \xi (x - Z_s) } ds d\xi.
\end{align}
As we mentioned in the beginning of this section, $Z$ is selfsimilar of index $H$, then its local time at level $x$ also has some selfsimilarity properties in time with index $1-H$, but with a different level as stated below. More precisely, one has, for every $c>0$:
\begin{align}
	\label{eq:local_selfsimilar}
	(L(x,ct))_{t\geq 0, x\in\mathbb{R}} 
	\stackrel{d}{=} 
	c^{1-H} (L(c^{-H}x,t))_{t\geq 0, x\in\mathbb{R}}.
\end{align}
Indeed, for every $c>0$, $t \geq 0$ and $x \in \R$, one has 
\begin{align*}
	L(x, ct) & = \frac{1}{2 \pi} \int_\R \int_0^{ct} e^{i \xi (x - Z_s) } ds d\xi   = c \,\frac{1}{2 \pi} \int_\R \int_0^{t} e^{i \xi (x - Z_{cs}) } ds d\xi\\
	&  \stackrel{d}{=} c \,\frac{1}{2 \pi} \int_\R \int_0^{t} e^{i \xi (x - c^HZ_s) } ds d\xi   = c^{1-H} \, \frac{1}{2 \pi} \int_\R \int_0^{t} e^{i \xi (c^{-H}x - Z_s) } ds d\xi  = L(c^{-H}x,t).
\end{align*}
Moreover, a recent result~\cite[Theorem 1.4]{knsv-2020} describes the scaling behavior of the local time of $Z$: 
\begin{prop}The local time $L(x, [0,t])$ is jointly continuous with respect to $(x,t)$ and has finite moments. For a finite closed interval $I \subset (0, \infty)$, let $L^*(I) = \sup_{x\in\R}L(x,I)$. There exist positive constants $C_1$ and $C_2$ such that, almost surely, for any $s \in I$,
	\begin{align}
		\label{eq:main-fixed-s}
		\limsup_{r\to 0} \frac{L^*([s-r,s+r])}{r^{1-H}(\log\log r^{-1})^{\cexp}} \leq C_1,
	\end{align}
	and
	\begin{align}
		\label{eq:main-sup-s}
		\limsup_{r\to 0}\sup_{s\in I} \frac{L^*([s-r,s+r])}{r^{1-H}(\log r^{-1})^{\cexp}} \leq C_2.
	\end{align}
	\label{HolderCond}
\end{prop}

Furthermore, we can establish the following property which is key in the study of the classical Hausdorff dimension of the level sets.

\begin{prop}\label{prop:sup_holder_lt} For $\beta \in \left( 0, \frac{1}{2} \left( \frac{1}{H} - 1\right) \right)$, 
	\begin{align}
		\notag \PP \left( \sup_{x \in [-1, 1] \setminus \{0\} } \frac{ \left|L \left(0, \left[ \frac{1}{2}, 1\right]\right) - L \left(x, \left[ \frac{1}{2}, 1\right]\right)\right| }{|x|^\beta} < \infty \right) = 1. 
	\end{align} 
	
\end{prop}

\begin{proof} The result relies on a  celebrated  lemma due to Garsia, Rodemich and Rumsey~\cite{grr-1970},  as well as on the moment estimates for the local time in~\cite{knsv-2020}. First, let us recall the lemma from~\cite{grr-1970}:

	\begin{lemma} Let $\Psi(u)$ be a non-negative even function on $(- \infty, \infty)$ and $p(u)$ be a non-negative even function on $[-1,1]$. Assume both $p(u)$ and $\Psi(u)$ are non decreasing for $u \geq 0$. Let $f$ be continuous on $[0, 1]$ and suppose that
		\begin{align}
			\notag \int_0^1 \int_0^1 \Psi\left( \frac{f(u) - f(v) }{p(u - v)} \right)du dv \leq B < \infty.
		\end{align}
		Then, for all $x, y \in [0, 1]$, 
		\begin{align}
			\notag |f(x) - f(y) | \leq 8 \int_0^{|x - y|} \Psi^{-1} \left( \frac{4B}{u^2} \right) dp(u).
		\end{align}
	\end{lemma}

	Let $\Psi(u) = |u|^p$ and $p(u) = |u|^{\alpha + 1/p}$ where $\alpha \geq 1/p$ and $p \geq 1$. Then for any continuous $f$ and $x \in [0,1]$, 
	$$
	|f(x)-f(y)|^p \leq C_{\alpha,p}|x - y|^{\alpha p -1}\int_{[0,1]^2}|f(r)-f(v)|^p |r-v|^{-\alpha p - 1}dr dv.
	$$
	Here the constant $C_{\alpha,p}$ is given by
	$
	C_{\alpha,p} = {4\cdot 8^p\big(\alpha + p^{-1}\big)^{p}\big(\alpha -p^{-1}\big)^{-p}}.
	$
	Thus, for fixed $\alpha$ and large enough $p$, we have
	$
	C_{\alpha,p} \leq C(\alpha)^p,
	$
	where $C(\alpha)>0$ is a constant that depends on the chosen $\alpha$.  We apply this
	to $f(x) = L\left(2x-1, \left[ \frac{1}{2}, 1\right]\right)$:
	\begin{align}
		\notag & \sup_{x \in [-1,1]\setminus \{0\}} \frac{\left|L\left(x, \left[ \frac{1}{2}, 1\right]\right) - L\left(0, \left[ \frac{1}{2}, 1\right]\right) \right|^p}{\left(\frac{x}{2}\right)^{\alpha p -1} }\\
		\notag \leq & \quad C'^p \int_{[0,1]^2} \left|L\left(r, \left[ \frac{1}{2}, 1\right]\right) - L\left(v, \left[ \frac{1}{2}, 1\right]\right) \right|^p |r-v|^{-\alpha p - 1}dr dv.
	\end{align}
	Using the moment bounds for the occupation density established in~\cite[Theorem 3.1]{knsv-2020}, one has
	\begin{align}
		\notag & \E \left[ \sup_{x \in [-1,1]\setminus \{0\}} \frac{\left|L\left(x, \left[ \frac{1}{2}, 1\right]\right) - L\left(0, \left[ \frac{1}{2}, 1\right]\right) \right|^p}{\left(\frac{x}{2}\right)^{\alpha p -1} }\right] \\
		\notag \leq & \quad C'^p \int_{[0,1]^2}  \frac{c(\gamma, H)^p p^{p2H(1 + \gamma) }2^{\gamma p} |r - v|^{\gamma p} }{2^{(1 - H - \gamma H)} }  |r-v|^{-\alpha p - 1}dr dv,
	\end{align}
	where $\gamma \in \left[ 0, \frac{1-H}{2H} \right)$ and $c(\gamma, H) > 0$ is a constant depending only on $\gamma$ and $H$. Let $\alpha = \gamma/2$ and $p > 4 / \gamma$. Then, 
	\begin{align}
		\notag  \E \left[ \sup_{x \in [-1,1]\setminus \{0\}} \frac{\left|L\left(x, \left[ \frac{1}{2}, 1\right]\right) - L\left(0, \left[ \frac{1}{2}, 1\right]\right) \right|^p}{\left(\frac{x}{2}\right)^{\alpha p -1} }\right] \leq C(\gamma, H, p), 
	\end{align}
	where $C(\gamma, H, p) > 0$ is a constant that depends on $\gamma, H$ and $p$.

	Fatou's lemma implies that 
	\begin{align}
		\notag \PP \left( \sup_{x \in [-1, 1]/\{0\} } \frac{ \left|L \left(0, \left[ \frac{1}{2}, 1\right]\right) - L \left(x, \left[ \frac{1}{2}, 1\right]\right)\right| }{|x|^\beta} < \infty \right) = 1, 
	\end{align} 
	as desired.
	
\end{proof}

Finally, the local time is H\"older continuous in both time and space~\cite[Corollary 3.2]{knsv-2020}. In particular: 
\begin{prop}
	\label{HolderLocalTime}
	For every $x \in \R$,almost surely, the local time $L(x,t)$ is H\"older continuous in t of order $\alpha$ for every $\alpha \in [0,1-H)$.
\end{prop}

\section{Image sets}\label{sec:image_set}

The present section is dedicated to the study of intermediate dimensions and profiles. To make a comparison, we recall the more popular packing dimensions and profiles. 

\subsection{Packing dimensions}\label{sec:packing_dim}
First, recall the definition of the packing dimension. For any $\alpha > 0$, the $\alpha-$dimensional packing measure of $E \subset \R^N$ is
\begin{align}
\notag \sap{E}\coloneqq \inf \lcl \sum_n \saP{E_n} : E \subseteq \bigcup_n E_n \rcl,
\end{align} 
where for $E \subset \R$, 
\begin{align} 
\notag \saP{E} \coloneqq \lim_{\ep \to 0} \sup \lcl \sum_i (2r_i)^\alpha : \overline{B}(x_i, r_i) \text{ are disjoint }, x_i \in E, r_i < \ep \rcl.
\end{align}
The packing dimension of $E$ is 
\begin{align}
\label{df:packingdimesion}	
\dpack{E} \coloneqq \inf \{ \alpha > 0 : \sap{ E} = 0 \}
\end{align}
and the packing dimension of a Borel measure $\mu$ on $\R^N$ is defined by 
\begin{align}
	\notag \dpack{\mu} \coloneqq \inf \{ \dpack{E} \, : \, \mu(E)>0 \mbox{ and } E \subset \R^N \mbox{ is a Borel set}\}.
\end{align}
Next, we recall the concept of packing dimension profiles first conceived by Falconer and Howroyd in~\cite{falconer1997packing} and~\cite{howroyd2001box}. For finite Borel measures $\mu$ on $\R^N$ and for any $s>0$, let
\begin{align*}
	F_s^\mu(x,r)= \int_\R \psi_s \left(\frac{x-y}{r}\right)d \mu(y),
\end{align*}  
be the potential with respect to the kernal $\psi_s \left(x\right)= \min \left\{1, \lVert x \rVert^{-s}\right\}$,$\forall x \in \R^N$.
\\ The packing dimension profile of $\mu$ is defined as follows 
\begin{align*}
	\dpackprof{s}{\mu}= \sup \left\{ \beta \geq 0: \, \liminf_{r \rightarrow 0} \dfrac{F_s^\mu(x,r)}{r^\beta}=0 \mbox{ for } \mu-a.e. \, x \in \R^N\right\}.
\end{align*}
Now for any Borel set $E \subset \R^N$, we define $\mathcal{M}_c^+(E)$ to be the family of finite Borel measures on E with compact support in E. Then
\begin{align*}
	\dpack{E} = \sup \left\{\dpack{\mu} \, : \, \mu \in \mathcal{M}_c^+(E)\right\}.
\end{align*}
Motivated by  this, Falconer and Howroyd \cite{falconer1997packing} define s-dimensional packing dimension profile of $E \subset \R^N$ by
\begin{align*}
	\dpackprof{s}{E}= \sup \left\{\dpackprof{s}{\mu} \, : \, \mu \in \mathcal{M}_c^+(E)\right\}.
\end{align*} 
It is easy to see that $0 \leq \dpackprof{s}{E} \leq s$ and for any $s \geq N$, $\dpackprof{s}{E} = \dpack{E}$. 
\subsection{Intermediate dimensions}\label{sec:inter_dim}
 For a bounded and non-empty set $E \subset \R^N$, $\theta \in (0,1]$ and $s \in [0,N]$, define  
\begin{align}
	H^s_{r,\theta}(E) = \inf \left\{\sum_i |U_i|^s \, : \, \left\{U_i\right\}_i  \mbox{ is a cover of } E \mbox{ such that } r \leq |U_i| \leq r^\theta \mbox{ for all } i\right\}. 
\end{align}
In particular, for $\theta=0$, $H^s_{r,0}(E)$ is the $s$-dimensional Hausdorff measure of E. Now, the intermediate dimensions are defined as in~\cite{falconer2004fractal}: 
\begin{defn} 
	\label{df:intermediate dimension}
	Let $E \subset \R^N$ be bounded. For $0 \leq \theta \leq 1$, the lower $\theta$-intermediate dimension is
\begin{align}
	\dltheta E = \mbox{ the unique } s \in [0,N] \mbox{ such that } \liminf_{r \rightarrow 0} \dfrac{\log H^s_{r,\theta}(E) }{-\log r}= 0.
\end{align}
Similarly,  the upper $\theta$-intermediate dimension of $E$ is defined by 
\begin{align}
	\dutheta E = \mbox{ the unique } s \in [0,N] \mbox{ such that } \limsup_{r \rightarrow 0} \dfrac{\log H^s_{r,\theta}(E) }{-\log r}= 0.
\end{align}
When $\dltheta E= \dutheta E$, we refer to the $\theta$-intermediate dimension $\dtheta E = \dltheta E= \dutheta E$.
\end{defn}
Thus, the classical Hausdorff \eqref{eq:classical_hausdorff} and box dimensions \eqref{eq:lower-box-dim}, \eqref{eq:upper-box-dim} can be viewed  as the extremes of a continuum of dimensions with increasing restrictions on the relative sizes of covering sets. Indeed, for every bounded $E \subset \R$, 
\begin{align}
\notag \underline{\dim}_0 E = \overline{\dim}_0 E = \chaus E, \quad \underline{\dim}_1 E = \dlbox{E} \quad \mbox{ and } \quad \overline{\dim}_1 E = \dubox{E}.
\end{align} 

Moreover, the intermediate dimensions can be defined in terms of capacities with respect to an appropriate kernel denoted by  $\phi_{r,\theta}^{s,m}$ (see~\cite{burrell2019projection}). For each collection of parameters $\theta \in (0,1]$, $0<m\leq 1$, $0 \leq s \leq m$ and $0<r<1$, let $\phi_{r,\theta}^{s,m}\, : \, \R^N \rightarrow \R$ be the function
\begin{align}
	\phi_{r,\theta}^{s,m}(x) \coloneqq \left\{ \begin{matrix} 1 & 0 \leq |x| < r, \\
		\left(\frac{r}{|x|}\right)^s & r \leq |x| < r^{\theta}, \\ 
		\frac{r^{\theta(m-s)+s}}{|x|^m} & r^\theta \leq |x|. 
	\end{matrix} \right.
\end{align}
Using this kernel we define the \textit{capacity} of a compact set $E \subset \R^N$ as 
\begin{align}
	C_{r,\theta}^{s,m}(E) \coloneqq \left(\inf_{\mu \in \mathcal{M}(E)} \int \int \phi_{r,\theta}^{s,m}(x-y) d\mu(x) d\mu(y)\right)^{-1},
\end{align}
where $\mathcal{M}(E)$ is the set of probability measures supported in $E$.  

Now for $0< m \leq N$,  the \textit{lower intermediate dimension profiles} of $E \subset \R^N$ are 
\begin{align}
\label{eq:dltp}
	\dlthetaprof{m}{E}  \, =\, \left(\mbox{the unique } s \in [0,m] \mbox{ such that } \liminf_{r \rightarrow 0} \frac{\log C_{r,\theta}^{s,m}(E)}{- \log r}=s\right),
\end{align}
and the \textit{upper intermediate dimension profiles} are
\begin{align}
\label{eq:dutp}
	\duthetaprof{m}{E} \, =\, \left(\mbox{the unique } s \in [0,m] \mbox{ such that } \limsup_{r \rightarrow 0} \frac{\log C_{r,\theta}^{s,m}(E)}{- \log r}=s\right).
\end{align}

The intermediate dimension profiles are increasing in $m$ and for $E \subset \R^N$,
\begin{align}
\notag \dlthetaprof{N}{E} = \dltheta E \quad \mbox{ and } \quad \duthetaprof{N}{E} = \dutheta E.
\end{align}

We note that originally the definitions of capacities and profiles above were established for $E \subset \R^N$ and integers $m \in (0, N]$. However, the recent result~\cite[Lemma 2.1]{burrell2020dimensions}, allows one to work with the version stated above. In fact, our first main result Theorem~\ref{thm:images} is an extension of a similar result in~\cite{burrell2020dimensions} obtained for the index-$\alpha$ fractional Brownian motion.   We proceed with the proof of Theorem~\ref{thm:images}

\subsection{Proof of Theorem~\ref{thm:images}}
	Let $\theta \in (0,1]$. We first state two results due to Burrell~\cite{burrell2020dimensions}. The first one establishes an upper bound for the intermediate dimensions of H\"older images using dimension profiles: 
	\begin{lemma}\cite[Theorem 3.1]{burrell2020dimensions}
		\label{UpperBound}
		Let $E \subset \R$ be a compact, $\theta \in (0,1]$, $m \in \{1,...,n\}$ and $f: E \rightarrow \R$. If there exist $c>0$ and $0<\alpha \leq 1$ such that 
		\begin{align*}
			|f(x)-f(y)| \leq c|x-y|^\alpha,
		\end{align*}
		for all $x,y \in E$, then 
		\begin{align*}
			\dltheta{f(E)} \leq \frac{1}{\alpha} \dlthetaprof{ \alpha}{E} \quad \mbox{ and } \quad \dutheta{f(E)}\leq \frac{1}{\alpha} \duthetaprof{ \alpha}{E}.
		\end{align*}
	\end{lemma}
	The second result gives a lower bound for the intermediate dimensions of image of a compact set $E$ under measurable functions satisfying certain properties:
	\begin{lemma} \cite[Theorem 3.3]{burrell2020dimensions}
		\label{LowerBound}
		Let $E \subset \R$ be a compact, $\theta \in (0,1]$, $\gamma>1$ and $s \in [0,1)$. If $f : \Omega \times E \rightarrow \R$ is a random function such that for each $\omega \in \Omega$, $f(\omega,.)$ is a continuous measurable functions and there exists $c>0$ satisfying 
		\begin{align*}
			\mathbb{P} \left(\{\omega \in \Omega \,: \, |f(\omega,x)- f(\omega,y)|\leq r \}\right) \leq c \phi_{r^\gamma,\theta}^{1/\gamma, 1/\gamma}(x-y),
		\end{align*} 
		for all $x, y \in E$ and $r>0$, then 
		\begin{align*}
			\dltheta{f(\omega,E)} \geq \gamma \dlthetaprof{\alpha}{E}\quad \mbox{ and } \quad \dutheta{f(\omega,E)} \geq \gamma \duthetaprof{ \alpha}{E},
		\end{align*} 
		for almost all $\omega \in \Omega$.
	\end{lemma}
		Now let $0<\epsilon<H<1$. The Rosenblatt process $Z$ has H\"older continuous paths in time of order $H-\epsilon$, see~\cite[Propostion 3.5]{tudor-2013}, and so there exists, almost surely, $M>0$ such that 
		\begin{align*}
			|Z_s-Z_t| \leq M |s-t|^{H-\epsilon},
		\end{align*}
		for all $s,t \in E$. 
		In addition by Proposition \ref{prop:rosen_density}\ref{prop:unimodal}, the density function $f$ of $Z_1$ is continuous and $f(0)>0$. Then for all $s,t \in E$ and $r>0$, one has 
		\begin{align*}
			\mathbb{P} \left(|Z_s -Z_t| \leq r\right) = 
			 \mathbb{P} \left(|Z_1| \leq \dfrac{r}{|s-t|^H}\right)  \leq 
			  4f(0) \dfrac{r}{|s-t|^H}  =
			 4f(0) \, \phi_{r^{1/H},\theta}^{H,H}(s-t).
		\end{align*}
		Now since the profiles are monotonically increasing, by Lemmas~\ref{UpperBound} and~\ref{LowerBound}, one has almost surely
		\begin{align*}
			\frac{1}{H} \dlthetaprof{H}{E} \leq \dltheta {Z(E)} \leq \frac{1}{H-\epsilon} \dlthetaprof{H-\epsilon}{E} \leq \frac{1}{H-\epsilon} \dlthetaprof{H}{E},
		\end{align*}
		and
		\begin{align*}
			\frac{1}{H} \duthetaprof{H}{E} \leq \dutheta {Z(E)} \leq  \frac{1}{H-\epsilon} \duthetaprof{H-\epsilon}{E} \leq \frac{1}{H-\epsilon} \duthetaprof{H}{E}.
		\end{align*} 
		Letting $\epsilon \rightarrow 0$ establishes the result.

\section{Level sets}\label{sec:level_set}

The present section is devoted to the  proof of Theorem~\ref{thm:levels}. First, we establish~\eqref{eq:levelset_classical} and \eqref{eq:levelset_packing} - the result regarding the $\theta$-intermediate dimensions and the packing dimension. Recall that the definition of $\dtheta{E}$ and $\dpack{E}$ for $E \subset \R$ are given in definition~\ref{df:intermediate dimension} and \ref{df:packingdimesion} respectively. 
\\ Note that the techniques employed in this section apply for the fractional Brownian motion case. As mentioned earlier,~\cite[Theorem 5]{ayache-xiao-2005} establishes $\dbox{\mathcal{L}_X(x) \cap [\epsilon,1]} \leq 1- H$ and $\chaus{\mathcal{L}_X(x)\cap [\epsilon,1]} = 1-H$ was shown in~\cite{falconer2004fractal}. Thus from the defintion of the $\theta$-intermediate dimensions (see~\ref{sec:inter_dim}) $\dtheta{\mathcal{L}_X(x)\cap [\epsilon,1]} = 1-H$, as well. Relevant results about the local time can be found in~\cite{xiao1997}, which allows us to establish $\dpack{\mathcal{L}_X(x)} = 1-H$. 

\begin{proof}[Proof of~\eqref{eq:levelset_classical}]
Let $\theta \in [0,1]$. Recall that for any set $E \subseteq \R$, one has 
\begin{align*}
	&\chaus{E} \leq \dutheta{E} \leq \dltheta{E} \leq \dubox{E}, \, \mbox{and} \\  
	&\chaus{E} \leq \dpack{E} \leq \dubox{E}.
\end{align*}

It is enough to show that $\dubox{\ls\cap [\ep, 1]} \leq 1-H$ and $\chaus{\ls \cap [\ep, 1]} \geq 1 - H$ with probability one. Starting with the upper bound, we follow the technique  used for~\cite[Theorem 5]{ayache-xiao-2005} - an upper bound result for the classical Hausdorff dimension of level sets associated to fractional Brownian sheet. But in fact, the covers used are of equal length and so this technique gives  an upper bound for the Box dimension. 

For $n \geq 1$ we cover $[\ep, 1]$ by  $\lceil n^{1/H} \rceil$ subintervals $R_{n,\ell}$  of length $n^{-1/H}$, with $\ell \in \{ 1, 2, \ldots, \lceil n^{1/H} \rceil\}$. Let $0 < \delta < 1$ be fixed and $\tau_{n, \ell}$ be the left endpoint of the interval $R_{n, \ell}$. We first bound the probability $\PP (x \in Z(R_{n, \ell}))$:
\begin{align}
\notag \PP (x \in Z(R_{n, \ell}) ) \leq & \quad \PP (\sup_{t \in R_{n, \ell}} | Z_t - Z_{\tau_{n, \ell} } | \leq n^{- (1 - \delta)}, x \in Z(R_{n, \ell})) \\
\notag & \quad + \PP (\sup_{t \in R_{n, \ell}} | Z_t - Z_{\tau_{n, \ell} } | \geq n^{- (1 - \delta)}) \\
\notag \leq  & \quad  \PP (|Z_{\tau_{n, \ell}} -x | \leq n^{- (1 - \delta)}) +C_1\exp ( - c_1 n^{- (1 - \delta)} / n^{-1} ) \\
\label{eq:prob_xrnl} \leq & \quad  C_2 n^{- (1 - \delta)} +  C_1 \exp (-c_1 n^{\delta})  = O( n^{- (1 - \delta)}),
\end{align}
where we have used Proposition~\ref{prop:sup-tail}, and the fact that the density of $Z_t$ is continuous. 

We can cover the set $\ls \cap [\ep, 1]$ by a sequence of intervals $R'_{n, \ell}$ with $R'_{n, \ell} = R_{n, \ell}$ if $x \in Z(R_{n, \ell})$ and $R'_{n, \ell} = \emptyset$, otherwise, for $\ell \in \{ 1, 2, \ldots, \lceil n^{1/H} \rceil\}$. We need to show that
\begin{align} 
\label{eq:haus_measure_eta} \E \left[ \sum_{\ell = 1}^{\lceil n^{1/H} \rceil} |R'_{n, \ell}|^\eta \right] < \infty,
\end{align}
for  $\eta = 1 - H( 1 - \delta)$ and arbitrary $\delta > 0$. In turn this would imply by Fatou's  lemma that $\dubox{\ls \cap [\ep, 1]}  \leq \eta$ almost surely. Then, letting $\delta \to 0$ yields the upper bound on the upper Box dimension. 

We   establish~\eqref{eq:haus_measure_eta}:
\begin{align} 
\notag & E \left[ \sum_{\ell = 1}^{\lceil n^{1/H} \rceil} |R'_{n, \ell}|^\eta \right] \leq   \E \left[ \sum_{\ell = 1}^{\lceil n^{1/H} \rceil}  \left(n^{- 1/H}\right)^\eta \mathbf{1}_{x \in Z(R_{n, \ell})}\right] \\
\notag \leq & \quad c n^{1/H - 1/H( 1 - H(1 - \delta) ) - (1 - \delta) } = c,
\end{align} 
where the last inequality follows from the bound~\eqref{eq:prob_xrnl} on $\PP (x \in Z(R_{n, \ell}))$.

For the lower bound we first recall a relation between the H\"older regularity and the Hausdorff dimension.

\begin{prop}[Theorem 27 in~\cite{dozzi-2003}] Let $[u, v] \subset \R$ be a finite interval and $f : [u,v] \to \R$ be a continuous function with occupation density denoted by $L$. Suppose that $L$ satisfies a H\"older condition of order $\gamma \in ( 0, 1)$ (in the set variable). Then $\chaus {f^{-1}_{[u, v]}(x)} \geq \gamma$ for all $x \in \R$ such that $L(x, [u, v]) \neq 0$.
\end{prop}

A H\"older regularity condition for the local time of the Rosenblatt process was recently obtained in~\cite{knsv-2020}. In particular, the following holds:

\begin{prop}[Theorem 1.4 in~\cite{knsv-2020}]  Let $(Z_t)_{t \geq 0}$ be a Rosenblatt process with $H \in \left(\frac{1}{2}, 1\right)$ and local time $L$. For a finite closed interval $I \subset (0, \infty)$, there exists a constant $C > 0$ such that almost surely,
\begin{align}
\notag \limsup_{r \to 0} \sup_{s \in I} \frac{ \ds \sup_{x \in \R} L(x, [s-r, s+r] )} {r^{1 - H} |\log r |^{2H}} \leq C.
\end{align}
\end{prop}

Therefore, the occupation density of the Rosenblatt process satisfies a H\"older condition in the set variable of order $\gamma$ for all $\gamma < 1-H$, and thus $\chaus {\ls \cap [\ep, 1]} \geq 1- H$. 
\end{proof}

Before we establish the second part~\eqref{eq:levelset_as} of Theorem~\ref{thm:levels} we recall some definitions and properties regarding the macroscopic Hausdorff dimension. Of special interest is a relation between $\dhaus \eg$ and $\dhaus \ls$ which eases the proofs of both~\eqref{eq:levelset_as} and~\eqref{eq:haus_as}.

\subsection{Macroscopic Hausdorff dimension}\label{sec:macro_haus}

To set up the notation as in~\cite{khoshnevisan-xiao-2017, khoshnevisan-kim-xiao-2017}, consider the intervals $S_{-1} =[0,1/2)$ and $S_n = [2^{n-1},2^n)$ for $ n \geq 0$. For $E \subset \mathbb{R}^+$,  we define the set of {\it proper covers} of $E$ restricted to $S_n$ by 
\begin{align*}
\mathcal{I}_n(E) = \left\{ 
\begin{matrix}
\left\{I_i\right\}_{i=1}^{m}\, : & I_i=[x_i,y_i] \, \mbox{with} \, x_i,y_i \in \mathbb{N}, \, y_i>x_i,  \\ & I_i \subset S_n \, \mbox{ and } \, E \cap S_n \subset \bigcup_{i=1}^{m} I_i.
\end{matrix} \right\}.
\end{align*}
For any set $E \subset \mathbb{R}^+$, $\rho \geq 0$ and $n \geq -1$, define
\begin{align}
\notag \nu_\rho^n (E) \coloneqq \inf \lcl \sum_{i = 1}^m \lp \frac{ \diam{I_i}  }{2^n} \rp^\rho : \, \left\{I_i\right\}_{i=1}^{m} \in \mathcal{I}_n(E) \rcl,
\end{align} 
where $\diam{[a,b]}=b-a$.
\\ The macroscopic Hausdorff dimension of $E \subset \R_+$ is defined as:
\begin{align}
\notag 	\dhaus{E} \coloneqq \inf \lcl \rho \geq 0 : \sum_{n \geq 0} \nu_\rho^n (E) < \infty \rcl.
\end{align}

Next we establish a relation between~\eqref{eq:levelset_as}  of Theorem~\ref{thm:levels} and~\eqref{eq:haus_as} of Theorem~\ref{thm:sojourn} .

Recalling Definitions \ref{defLS} and \ref{defSS}, for a fixed $\gamma > 0$ and any $x \in \mathbb{R}$, the level set $\ls$ is ultimately included in $\eg$:
\begin{align*}
\ls \cap \left\{ t \geq |x|^{\frac1\gamma}\right\} \subset \eg.
\end{align*} 
The macroscopic Hausdorff dimension is left unchanged after the removal of any bounded subset. Then, almost surely, for every $x \in \mathbb{R}$, 
\begin{align}
\label{eq:macro_haus_incl} \dhaus \ls = \dhaus { \ls \cap \left\{ t \geq  |x|^{\frac1\gamma}\right\} } \leq \dhaus \eg.
\end{align}

\noindent Therefore, to prove~\eqref{eq:levelset_as} and~\eqref{eq:haus_as} it suffices to show  that the following two statements hold almost surely:
\begin{align}
 \label{eq:lset_lower} & \text{For any }  x \in \R, \dhaus{\ls} \geq 1 - H,\\
\label{eq:haus_upper} & \dhaus{\eg}  \leq  1 - H.  
\end{align}
The proof of~\eqref{eq:lset_lower} follows in the next subsection while~\eqref{eq:haus_upper} is established in Section~\ref{sec:dhaus_upper}.

\subsection{Lower bound for $\dhaus{ \ls}$}
\label{LowerBdDhausLx}

In this section we aim to find a lower bound for $\dhaus{ \ls}$. We first establish a result regarding macroscopic Hausdorff dimension in general.
\begin{lemma}
	\label{LemmaHD}
	Let $E \subset \R_+$ and suppose that there exist $M>0$ and $s\in [0,1]$ such that there exists a family of finite measures $\left\{\mu_n\right\}_{n \geq -1}$ on $S_n$ such that  for all intervals $I \subset S_n$, we have $\mu_n(I) \leq M \diam {I}^s$. If $\dhaus E = t$ for some $0 \leq t < s$, then $\sum_{n \geq -1} \dfrac{\mu_n \left(E \cap S_n\right)}{2^{ns}} < + \infty$.
\end{lemma}
\begin{proof}
	As $t < s$ and using the definition of macroscopic Hausdorff dimension we have $\nu^s(A) < + \infty$. 
	\\ Let $\left\{I_i\right\}_{i=1}^{m} \in \mathcal{I}_n(E)$, then 
	\begin{align*}
	\mu_n(E \cap S_n) \leq \sum_{i=1}^m \mu_n(I_i) \leq \sum_{i=1}^m M \diam {I_i}^s = 2^{ns}M \sum_{i=1}^m \lp \frac{\diam{I_i} }{2^n} \rp^s.
	\end{align*} 
	Then $\dfrac{\mu_n(E\cap S_n)}{2^{ns}} \leq M \nu_n^s(E\cap S_n)$ and so $\sum_{n \geq -1} \dfrac{\mu_n(A\cap S_n)}{2^{ns}} < + \infty$.	
\end{proof}

By Proposition \ref{HolderLocalTime}, the local time is H\"older continuous in $t$ of order $\alpha$ for every $\alpha \in [0,1-H)$. Now we will be using this property and the preceding lemma in order to get a lower bound for $\dhaus { \ls}$. To this end, fix $\alpha \in [0,1-H)$ and  introduce the following random variables 
\begin{equation}
Y_{n}^{x} = \dfrac{L\left(x,S_n\right)}{2^{n\alpha}}  \quad \mbox{and} \quad F^{x}_{N} = \sum_{n=1}^{N} Y_{n}^{x}.
\label{def}
\end{equation}
The random variables $\left(Y_{n}^{x} \right)_{n \geq -1}$ are positive, so $(F_{N}^{x})_{N \geq 1}$ is non-decreasing. We denote by $F_{\infty}^{x}$ its limit, i.e. $ F_{\infty}^{x} = \sum_{n=-1}^{\infty} Y_{n}^{x} \in [0, + \infty]$.

 As a direct consequence of Lemma \ref{LemmaHD}, there is a connection between $\dhaus {\mathcal{L}_x}$ and the r.v. $Y_{n}^{x}$. Indeed, for $n \geq -1$  consider the sequence of measures 
 \begin{align*}
 	\mu_n(I):= L(x, I), \, \mbox{for all } I \subset S_n,
 \end{align*}
 By Proposition \ref{HolderLocalTime}, there exists $M>0$ such that for all $n \geq -1$  a.s. 
 \begin{align*}
 	\mu_n\left(I\right) \leq M \diam {I}^\alpha, \, \mbox{for all } I \subset S_n.
 \end{align*}
 Now by Lemma \ref{LemmaHD}, a.s. for every $x \in \R$, $\dhaus {\mathcal{L}_x} \geq \alpha$ if 
 \begin{equation*}
 	\sum_{n \geq -1} \dfrac{\mu_n \left(\mathcal{L}_x \cap S_n\right)}{2^{n\alpha}} = F_\infty^x = + \infty.
 \end{equation*}
As a consequence, we see that $\dhaus { \mathcal{L}_x }\geq \alpha$ for all $x \in \mathbb{R}$ such that $F_{\infty}^{x} = + \infty$. Moreover in order to conclude the proof of Theorem~\ref{thm:levels}, it is enough to prove that for all $\alpha \in [0,1-H)$, a.s. for all $x \in \R$, $\dhaus {\mathcal{L}_x} \geq \alpha$ . Letting $\alpha \uparrow 1-H$ gives that a.s. for all $x \in \R$, $\dhaus{ \mathcal{L}_x }\geq 1-H$. Finally it remains  to check that 
$\mathbb{P}(\forall x \in \mathbb{R}, \, F_{\infty}^{x} = + \infty ) = 1$,  for all $\alpha \in [0,1-H)$. This is the object of the next proposition.

\begin{prop}
	\label{prop:uniform_as} Let $\alpha \in [0,1-H)$ and 
\begin{align}
\notag Y_n^x = \frac{L^x(S_n)}{2^{n\alpha} }, \text{ for } n \geq -1, \quad \text{ and } F_\infty^x = \sum_{n \geq -1} Y_n^x.
\end{align} 
Then,
	\begin{align}
	\mathbb{P}(\forall x \in \mathbb{R}, \, F_{\infty}^{x} = + \infty ) = 1.
	\label{eq:uniform_as}
	\end{align}
\end{prop}

\begin{proof} We follow the technique in~\cite{daw-2020}. For every $a > 0$, let
\begin{align}
\notag \widetilde{Y}_n^a = \inf_{x \in [-a, a]} Y_n^x, \text{ for } n \geq 1, \quad \text{ and } \quad \widetilde{F}_\infty^a  = \sum_{n \geq 1} \widetilde{Y}_n^a.
\end{align} 
Using the self-similarity property of the local time~\eqref{eq:local_selfsimilar}, for all $n \geq 0$, 
\begin{align}
\notag \widetilde{Y}_n^a = \inf_{x \in [-a, a]} Y_n^x \overset{d}{=} \inf_{x \in [-a, a]} Y_0^{2^{-nH}x} = \inf_{x \in [-2^{-nH}a, 2^{-nH}a]}  Y_0^x = \widetilde{Y}_0^{2^{-nH}a }.
\end{align} 

The proof now relies on the following technical result:

\begin{lemma}\label{lem:FZ}  For any $b > 0$, one has
\begin{align}
\notag \PP(\wt{F}_\infty^b = \infty) > 0.
\end{align}
\end{lemma}

\begin{proof}[Proof of Lemma~\ref{lem:FZ}] We first show that there exists $\ep > 0$ such that $\PP(Y_0^0 > \ep) > 0$. Recall that $Y_0^0 = L(0, [1/2, 1])$ and is non-negative. Thus, it is enough to show that $\E[ L(0, [1/2,1]) ] > 0$. Using the following representation of the local time, see~\cite[Chapter 10]{samorodnitsky-2016}, one gets 
 \begin{align}
\notag L(0, [1/2, 1]) = \lim_{\ep \to 0} \frac{1}{2\epsilon}\int_{1/2}^1 \mathbf{1}_{-\ep, \ep} (Z_t) dt.
\end{align}
Then using self-similarity of $Z$ and then Proposition~\ref{prop:rosen_density}\ref{prop:unimodal} with some constant $c_1 > 0$, one gets
\begin{align}
\notag  \E[ L(0, [1/2, 1]) ] = &  \lim_{\ep \to 0 }\frac{1}{2\epsilon} \int_{1/2}^1 \PP( Z_t \in [-\ep, \ep]) dt \\
\notag = & \lim_{\ep \to 0 } \frac{1}{2\epsilon}\int_{1/2}^1 \PP(Z_1 \in [- \ep t^{-H} , \ep t^{- H}]) dt \\
\notag \geq & \lim_{\ep \to 0 } \frac{1}{2\epsilon}\int_{1/2}^1 2 c_1 \ep t^{-H} dt\\
\notag = & \frac{c_1}{1 - H} (1 - (1/2)^{1 - H}).
\end{align} 
 Therefore, $\E[ L(0, [1/2,1]) ] > 0$ and thus $\PP(Z_0^0 > \ep) > 0$ for some $\ep > 0$. 

The rest of the proof is based on the following two facts:
\begin{enumerate}
\item For every $\ep > 0$ small enough, there exists $a \in \R^+$, such that:
\begin{align} 
\notag 0 < \PP(Y_0^0 > \ep) \leq 2 \PP(\wt{Y}_0^a > 0). 
\end{align} 
\item  For any $a, b > 0$, we have 
\begin{align}
\notag \PP(\wt{F}_\infty^b = \infty) \geq \PP (\wt{Y}_0^a). 
\end{align}
\end{enumerate} 
The statements above  correspond to Lemmas 8 and 9 in~\cite{daw-2020} and the proofs are identical as long as the following holds:
\begin{align}
\notag \PP \left( \sup_{x \in [-1, 1]/\{0\} } \frac{ \left|L \left(0, \left[ \frac{1}{2}, 1\right]\right) - L \left(x, \left[ \frac{1}{2}, 1\right]\right)\right| }{|x|^\beta} < \infty \right) = 1, 
\end{align} 
where $\beta \in \left( 0, \frac{1}{2} \left( \frac{1}{H} - 1\right) \right)$. In~\cite{daw-2020}, this property corresponds to Lemma 5 which is originally due to Geman in Horowitz~\cite[Theorem 26.1]{geman-horowitz-1980}. For the Rosenblatt case, the above is established in Proposition~\ref{prop:sup_holder_lt}.
\end{proof}

Using the result of Lemma~\ref{lem:FZ} we can establish that  $\PP(\wt{F}_\infty^b = \infty) =1$ if we can apply Blumental's 0-1 law. This is possible since
\begin{align}
\label{eq:event_germ} \lcl \wt{F}_\infty^b = \infty \rcl \in \bigcap_{M \geq 1} \sigma \lcl W_u : u < 2^{-(M-1)} \rcl,
\end{align}
where $(W_t)_{t \geq 0}$ is the standard Brownian motion. Indeed, the time inverted process $\wt{Z}_t = t^{2H} Z_{1/t}, t > 0$ is distributed as the Rosenblatt process (see~Proposition~\ref{prop:inverse_time}). Then, using the representation~\eqref{eq:local_rep}, the local time $L^x(S_n)$ is $\sigma \lcl \wt{Z}_u : u \leq 2^{- (n-1)} \rcl$-measurable. Moreover, 

\begin{align}
\notag \sigma \lcl \wt{Y}_n^b : n \geq M \rcl \subset \sigma \lcl \wt{Z}_u  : u \leq 2^{- (M - 1)} \rcl, 
\end{align} 
for $M \geq 1$ and thus
\begin{align}
\label{eq:event_germ_r}  \lcl \wt{F}_\infty^b = \infty \rcl \in \bigcap_{M \geq 1} \sigma \lcl \wt{Z}_u : u < 2^{-(M-1)} \rcl.
\end{align} 

At this point by~\eqref{filtrationB} (with $\tilde{Z}$ instead of $Z$), there exists standard Brownian motion $(W_t)_{t \geq 0}$ such that $\sigma \lcl \tilde{Z}_u : u \leq t \rcl \subset \sigma \lcl W_u : u \leq t \rcl$. This fact combined with~\eqref{eq:event_germ_r} establishes~\eqref{eq:event_germ}. Then using~\eqref{eq:event_germ} and the fact that $\PP(\wt{F}_\infty^b = \infty) > 0$ for all $b > 0$, one can apply Blumental's 0-1 law and thus gets that $\PP(\wt{F}_\infty^b = \infty) =1$, for all $b > 0$.

Finally, for every $b > 0$, 
\begin{align}
\notag & \PP (\forall x \in [-b, b ] : F_\infty^x = \infty)  = \PP \left( \inf_{x \in [-b, b]} F_\infty^x = \infty\right) = \PP \left( \inf_{x \in [-b, b]} \sum_{n \geq 1} Y_n^x = \infty\right) \\
\notag \geq & \quad \PP \left( \sum_{n \geq 1} \inf_{x \in [-b, b ]} Y_n^x = \infty \right) = \PP ( \wt{F}_\infty^b = \infty) = 1. 
\end{align}
Therefore, 
\begin{align}
\notag \PP (\forall x \in \R : F_\infty^x = \infty) = \lim_{b \to \infty} \PP( \forall x \in [-b, b] , F_\infty^x = \infty ) = 1, 
\end{align}
and~\eqref{eq:uniform_as} is established.
  
\end{proof}

Next, we establish \eqref{eq:levelset_packing}- the result regarding Packing dimension. Recall that $\dpack{\ls} \leq \dubox{\ls}=1-H$. It is enough to show that $\dpack{\ls} \geq 1-H$, which is the aim of the next section.

\section{Sojourn times}\label{sec:sojourn_set}

This section is dedicated to the proof of Theorem~\ref{thm:sojourn}.  We first establish~\eqref{eq:densities}. Recall the definitions of logarithmic and pixel densities.  For $E \subset \R^+$, the logarithmic density of $E$ is given by
\begin{align}
\notag	\dlog{E} \coloneqq  \limsup_{n \to \infty} \frac{\log_2 \leb{E \cap [ 1, 2^n]} }{n},
\end{align} 
where `Leb' is the one-dimensional Lebesgue measure. 

Let $ \pix{E} \coloneqq \lcl n \in N : \text{dist}(n, E) \leq 1 \rcl$. Then, the pixel density of $E$ is
\begin{align}
	\notag  \dpix{E} \coloneqq  \limsup_{n \to \infty} \frac{\log_2 \# \pix{E \cap [ 1, 2^n]} }{n}.
\end{align}

The two quantities are closely related, see~\cite{khoshnevisan-xiao-2017}:
\begin{align}
\label{eq:standard} \dlog{E} \leq \dpix{E}.
\end{align}

We want to show that for $\gamma \in [0,H)$, $\dpix{\eg} = \dlog{\eg} = \gamma + 1 - H$, almost surely. Our strategy is then to establish  that $\dpix{\eg} \leq \gamma + 1 - H$ and $\dlog{\eg} \geq \gamma + 1 - H$, almost surely.

\subsection{Upper bound for $\dpix{\eg}$}
Our goal is to obtain an upper bound for $\# \pix{\eg} \cap [1, 2^n]$ that holds with probability $1$ for all large $n$. We first study the expectation
\begin{align}
\notag \E[\pix{\eg} \cap [1, 2^n]] = & \sum_{m=1}^{2^n} \PP \lp \exists s \in [ m - 1, m + 1], |Z_s| \leq s^\gamma \rp \\
\notag = & \sum_{m = 1}^{2^n} \PP \lp \exists s \in \lc 1 - \frac{1}{m}, 1 + \frac{1}{m} \rc, |Z_s| \leq s^\gamma m ^{\gamma - H} \rp\\
\notag = & \sum_{m = 1}^{2^n} \PP \lp \exists s \in \lc 1 - \frac{1}{m}, 1  \rc, |Z_s| \leq s^\gamma m ^{\gamma - H} \rp \\ \notag & + \PP \lp \exists s \in \lc 1, 1 + \frac{1}{m} \rc, |Z_s| \leq s^\gamma m ^{\gamma - H} \rp\\
\label{eq:aem} \leq & \sum_{m =1}^{2^n} (A_{1/m}^- + A_{1/m}^+),
\end{align}
where 
\begin{align}
\notag A_{\ep}^- \coloneqq & \PP (\exists s \in [1 - \ep, 1], |Z_s| \leq \ep^{H - \gamma}), \\
\notag A_{\ep}^+ \coloneqq & \PP (\exists s \in [1 , 1 + \ep], |Z_s| \leq 2\ep^{H - \gamma}).
\end{align}

\begin{lemma}\label{lem:max_As} There is a universal constant $c > 0$, such that, for every $\ep$ small enough, 
	\begin{align}
	\label{eq:max_As} \max(A_\ep^-, A_\ep^+) \leq c \ep^{H - \gamma}.
	\end{align}
\end{lemma}

\begin{proof} Consider $A_\ep^-$ first. We have
	\begin{align}
	\notag A_\ep^- \leq & \quad \PP (\exists s \in [1- \ep, 1],  |Z_s| \leq \ep^{H - \gamma}, |Z_1| \leq 2 \ep^{H - \gamma}) \\ 
	\notag & \quad \quad + \PP (\exists s \in [1- \ep, 1],  |Z_s| \leq \ep^{H - \gamma}, |Z_1| \geq 2 \ep^{H - \gamma}) \\ 
	\label{eq:aep_1} \leq &\quad \PP (|Z_1| \leq 2 \ep^{H - \gamma}) + \PP(\exists s \in [1- \ep, 1], |Z_s - Z_1 | \geq \ep^{H - \gamma}).
	\end{align}
	
To bound the first term on the right-hand side above, we use Proposition~\ref{prop:rosen_density}\ref{prop:unimodal}, i.e,  the density function $f$ of $Z_1$  is continuous and $f(0) > 0$. Then one can show, for instance, that for $\ep > 0$ small enough,
\begin{align}
\label{eq:aep_2}  \PP (|Z_1| \leq 2 \ep^{H - \gamma}) \leq 4f(0) \ep^{H -\gamma}.
\end{align}

	We are left to  study the term $\PP(\exists s \in [1- \ep, 1], |Z_s - Z_1 | \geq \ep^{H - \gamma})$. Write
\begin{align}
\notag &  \PP(\exists s \in [1- \ep, 1], |Z_s - Z_1 | \geq \ep^{H - \gamma}) \\
\notag \leq & \quad \PP(\sup_{s \in [1- \ep, 1+\ep]} |Z_s - Z_1 | \geq \ep^{H - \gamma}) \\
\label{eq:aep_3} \leq & \quad C \exp \left( - c_1 \ep^{- \gamma} \right),
\end{align}
where the last inequality follows from Proposition~\ref{prop:sup-tail} and $C, c_1> 0$ are constants depending only on $H$. Note that $  \exp( - c_1 \ep^{- \gamma}) = O(\ep^\delta)$, for any $\delta > 0$ if any $\ep$ is small enough.

Finally, for $\ep$ small enough, combining~\eqref{eq:aep_3} and~\eqref{eq:aep_2} in~\eqref{eq:aep_1} yields the bound of~\eqref{eq:max_As} for $A_\ep^-$. 
	
	Same arguments as above can be applied to $A_\ep^+$ to get an equivalent bound and establish~\eqref{eq:max_As}.
\end{proof}

Next, applying Lemma~\ref{lem:max_As} in~\eqref{eq:aem} yields, for some absolute constant $C > 0$,
\begin{align}
\notag \E[\#\pix{ \eg} \cap [1, 2^n] ]  \leq  2 C \sum_{m =1}^{2^n} m^{\gamma - H} = O\lp 2^{n (\gamma + 1 - H)}\rp.
\end{align}
Choose $\rho > \gamma + 1 - H$. Then, 
\begin{align}
\notag \sum_{n \geq 1} \PP \lp \# \pix{ \eg} \cap [1, 2^n]  > 2^{n \rho} \rp \leq C \sum_{n \geq 1} \frac{2^{n(1 + \gamma - h)}}{2^{n \rho}} < \infty.  
\end{align}
By the Borel-Cantelli lemma, with probability one, 
\begin{align}
\notag \# \pix{\eg} \cap [1, 2^n]  \leq 2^{n \rho},
\end{align}
for every large enough $n$. Hence, $\dpix{\eg}\leq \rho$. Letting $\rho \downarrow \gamma + 1 - H$ yields $\dpix{\eg} \leq \gamma + 1 - H$.

\subsection{Lower bound for $\dlog{\eg}$}

Introduce  
\begin{align}
\notag S_\gamma([t_1,t_2]) = \leb{\{t_1 \leq s \leq t_2: |Z_s| \leq s^\gamma\}}, \mbox{ for all } 0 \leq t_1 \leq t_2.
\end{align}

We will prove that for infinitely many integers $n$, $S_\gamma([0,2^n]) \geq \frac{c}{2} 2^{n ( \gamma + 1 - H) }$, for any $c \in (0,1)$. This implies that $\dlog{\eg} \geq \gamma + 1 - H$ almost surely. Then using~\eqref{eq:standard}, we also obtain $\dpix{\eg} \leq \gamma + 1 - H$ and the proof of~\eqref{eq:densities} is completed.

First we show that for any $ c \in (0, 1)$, there is a constant $c' > 0$ such that
\begin{align}
\label{eq:bound_sg} \PP (S_\gamma( [0,2^n]) \geq c 2^{n(1 + \gamma - H)} ) \geq c'. 
\end{align}

By Paley-Zygmund inequality, for any $ c \in (0, 1)$, we have
\begin{align}
\label{eq:pz_ineq} \PP (S_\gamma( [0,2^n]) \geq c 2^{n(1 + \gamma - H)} ) \geq (1 - c) \frac{\E [ S_\gamma([0,2^n])]^2 }{\E [ S_\gamma([0,2^n])^2 ] }. 
\end{align}

The numerator can be rewritten as:
\begin{align}
\notag \E [S_\gamma([0,t]) ] = \int_0^t \PP (|Z_s | \leq s^\gamma) ds = \int_0^t \PP (|Z_1| \leq s^{\gamma - H} ) ds.
\end{align}

Now, we establish a lower bound for $\PP (|Z_1| \leq s^{\gamma - H} )$. Apply Proposition~\ref{prop:rosen_density}\ref{prop:unimodal} there is a constant $\alpha > 0$ such that for $s$ large enough, the density function of $Z_1$ is bounded below by $\alpha$ in $[-s^{\gamma - H}, s^{\gamma - H}]$. Therefore,
\begin{align}
\label{eq:bound_1st_moment} \PP (|Z_1| \leq s^{\gamma - H} )  \geq 2\alpha s^{\gamma - H} \quad \text{ and thus } \quad \E[S_\gamma([0,t])] \geq 2\alpha t^{1 + \gamma - H}.
\end{align}

We bound the second moment from above:
\begin{align}
\notag \E[S_\gamma([0,t])^2 ] = & \int \int_{[0, t]^2} \PP ( |Z_u | \leq u^\gamma, |Z_v| \leq v^\gamma) du dv \\
\notag = & \,t^2 \int \int_{[0, 1]^2} \PP \lp |Z_u| \leq u^\gamma t^{\gamma - H}, |Z_v| \leq v^\gamma t^{\gamma - H} \rp  du dv.
\end{align}

By Proposition~\ref{prop:rosen_density}\ref{prop:cont_density}, the density function $g_{u,v}$ of $(Z_u, Z_v)$ is continuous and tends to $0$ as $\lvert x \rvert \rightarrow \infty$. Therefore, 
\begin{align}
\label{eq:bound_2nd_moment}  \E[ S_\gamma([0,t])^2] & \leq \, t^2 \int \int_{[0, 1]^2} du dv \int \int_{\R^2} g_{u,v}(x,y) \, \mathds{1}_{ \begin{pmatrix}
		|x| \leq u^\gamma t^{\gamma - H} \\ |y| \leq v^\gamma t^{\gamma - H}
	\end{pmatrix}}dx dy \\
& \leq  C t^{2 + 2\gamma - 2H}. 
\end{align} 
Applying~\eqref{eq:bound_1st_moment} and~\eqref{eq:bound_2nd_moment} in~\eqref{eq:pz_ineq} yields~\eqref{eq:bound_sg}. 
Now, define the event 
\begin{align*}
	A_{n,\gamma} \coloneqq \left\{S_\gamma\lp \left[ \frac{c}{2} 2^{n(1 + \gamma - H)},2^n\right]\rp \geq \frac{c}{2} 2^{n(1 + \gamma - H)}\right\}. 
\end{align*}
By~\eqref{eq:bound_sg}, it is easy to see that $\PP\left(A_{n,\gamma}\right) \geq c' >0$. Moreover, by the definition of $A_{n,\gamma}$, one has $A_{n,\gamma} \subset \left\{S_\gamma\lp \left[ 0,2^n\right]\rp \geq \frac{c}{2} 2^{n(1 + \gamma - H)}\right\} $. Then it is enough to prove that $A_{n,\gamma}$ happens infinitely often which give us that $S_\gamma\lp \left[ 0,2^n\right]\rp \geq \frac{c}{2} 2^{n(1 + \gamma - H)}$ for infinitly many $n$.  To this end, let $A_\gamma$ be the event that $A_{n,\gamma}$ happens infinitely often. Recall that for any sequence of events $(A_i)_{i \geq 1}$, one has $\lim_{n \to \infty} \PP( \cup_{i \geq n} A_i)  = \PP (A_i \text{ i. o })$. In other words, one has  
\begin{align}
	\label{defAG}
	A_\gamma = \bigcap_{M \geq 1} \bigcup_{n \geq M} A_{n,\gamma}.
\end{align}
We know that $\PP\lp A_\gamma \rp (\geq c')$ is strictly positive. It remains to prove that it is in fact equal to 1. As in Section~\ref{LowerBdDhausLx}, such a conclusion will follow by using that the time inverted process $\tilde{Z}_t = t^{2H} Z_{1/t}$ is distributed as the Rosenblatt process (see Proposition~\ref{prop:inverse_time}). Now let $\tilde{S}_\gamma$ (resp. $\tilde{A}_{n,\gamma}$, $\tilde{A}_{\gamma}$) be the event analogous to $S_\gamma$ (resp. $A_{n,\gamma}$, $A_\gamma$), but associated to $\tilde{Z}$ instead of $Z$. So for any fixed integer $n \geq 0$, we have
\begin{align*}
	\tilde{S}_\gamma\lp \left[ \frac{c}{2} 2^{n(1 + \gamma - H)},2^n\right]\rp = \mbox{Leb} \lp \left\{ \frac{c}{2} 2^{n(1 + \gamma - H)} \leq s \leq 2^n: |t^{2H}Z_{1/s}| \leq s^\gamma \right\} \rp,
\end{align*}
which implies in return that $\tilde{A}_{n,\gamma} \in \sigma \left\{ Z_u \, : \, u \leq 2^{-n(1 + \gamma - H)} \right\}$. As a consequence, for all $M \geq 0$, one has 
\begin{align*}
	\left\{\tilde{A}_{n,\gamma} \, : \, n \geq M\right\} \in \sigma\left\{Z_u \, : \, u \leq 2^{-M(1 + \gamma - H)}\right\}.
\end{align*}
Recalling definition~\ref{defAG} of $A_\gamma$, we obtain that   
\begin{align*}
	\tilde{A}_{\gamma} \in \bigcap_{M \geq 1}\sigma \left\{\title{A}_{n,\gamma} \, : \, n \geq M\right\}.
\end{align*}
Using \eqref{filtrationB}, we deduce that 
\begin{align}
\notag \tilde{A}_{\gamma} \in \bigcap_{M \geq 1} \sigma ( B_u : u \leq 2^{-M(1 + \gamma - H)}), 
\end{align}
where $(B_t)_{t \geq 0}$ is the Brownian motion. Therefore, $\tilde{A}_\gamma$ is a tail event and $\PP(\tilde{A}_\gamma) = 0$ or $1$ by the Blumenthal $0-1$ law. Obviously, as $Z$ and $\tilde{Z}$ have the same distribution, then $\PP(\tilde{A}_\gamma)= \PP (A_\gamma) \geq c'>0$ and then $\PP(\tilde{A}_\gamma)= \PP (A_\gamma)= 1$ as desired.

\subsection{Upper bound for $\dhaus{\eg}$}\label{sec:dhaus_upper}
We now turn to the proof of~\eqref{eq:haus_as}. Following our discussion in Section~\ref{sec:macro_haus}, and in particular the relation~\eqref{eq:macro_haus_incl} between $\dhaus \eg$ and $\dhaus \ls$, it is enough to show~\eqref{eq:haus_upper}, i.e., for every $0 \leq \gamma < H$, 
\begin{align}
\notag \dhaus \eg \leq 1 - H, \text{ a.s.}
\end{align}

We follow the technique in \cite{Nourdin-Peccati-Seuret-2018}. Let us fix $0 \leq \gamma < H$, as well as $\eta > 0$ (as small as necessary). We
are going to prove that  $\dhaus{\eg} \leq  1 - H + \eta$. Letting $\eta$ tend to zero will then give the result.
Fix $\rho > 1-H + \eta$, our aim is to prove that $\dhaus{\eg} \leq \rho$. To this end, consider for every integer $n \geq 1$ and $i \in \left\{0, ...,\left\lfloor \dfrac{2^{n-1}}{2^{n\frac{\gamma}{H}}}\right\rfloor\right\}$ the intervals 
\begin{align*}
	I_{n,i} = [t_{n,i}, t_{n,i+1}) \mbox{ with } t_{n,i}= 2^{n-1} +i2^{n \frac{\gamma}{H}}.
\end{align*}
And the associated event
\begin{align*}
\mathcal{E}_{n,i} = \left\{ \exists t \in I_{n,i}\, :\,  |Z_t| \leq t^\gamma \right\}.
\end{align*}
Denote $\epsilon_{n,i} =2^{n \frac{\gamma}{H}}/t_{n,i}$ , so that $I_{n,i} = [t_{n,i}, t_{n,i}(1+\epsilon_{n,i}))$, and observe that the ratio between any two of the quantities $2^{n \left(\frac{\gamma}{H}-1\right)}, \, \epsilon_{n,i},$ and $t_{n,i}^{\frac{\gamma}{H}-1}$ are bounded uniformly with respect to $n$ and $i$. By self-similarity, we have that, when $n $ becomes large, 
\begin{align*}
\mathbb{P}\left(\mathcal{E}_{n,i}\right) 
& = \mathbb{P}\left(\exists \, t \in I_{n,i}\, : \,  |Z_t| \leq t^\gamma\right) \\
& = \mathbb{P}\left(\exists \, s \in \left[1,1+\epsilon_{n,i}\right]: |Z_{s. t_{n,i}}| \leq \left(s. t_{n,i}\right)^\gamma\right) \\
& = \mathbb{P}\left(\exists \, s \in \left[1,1+\epsilon_{n,i}\right]: |Z_{s}| \leq t_{n,i}^{\gamma-H}.s^\gamma\right) \\
& = \mathbb{P}\left(\exists \, s \in \left[1,1+\epsilon_{n,i}\right]: |Z_{s}| \leq 2t_{n,i}^{\gamma-H}\right) \\
& = \mathbb{P}\left(\exists \, s \in \left[1,1+\epsilon_{n,i}\right]: |Z_{s}| \leq c\epsilon_{n,i}^{H}\right) \\
& = \mathbb{P}\left(\exists \, s \in \left[1,1+\epsilon_{n,i}\right]: |Z_{s}| \leq \epsilon_{n,i}^{H-\eta}\right). 
\end{align*}
The last estimate holds because $\eta$ is a small positive real number and
$\epsilon_{n,i}$ tends to zero when $n$ becomes large. By Lemma \ref{lem:max_As}, we deduce that $\mathbb{P}\left(\mathcal{E}_{n,i}\right) \leq c \epsilon_{n,i}^{H-\eta}$ and so 
\begin{align*}
	\mathbb{P}\left(\mathcal{E}_{n,i}\right) \leq c 2^{n(\gamma-H)\frac{H-\eta}{H}}.
\end{align*}
Now observe that $\mathcal{E}_{n,i}$ is realized if and only if $\eg \cap I_{n,i} \neq \phi$. So, using the intervals $I_{n,i}$ as a covering of $\eg \cap S_n$, we obtain that
\begin{align*}
	\mathbb{E}\left[\nu_{\rho}^{n}(\eg) \right] \leq 
	& \mathbb{E} \left(\sum_{0}^{\lfloor 2^{n-1-n\frac{\gamma}{H}}\rfloor} \left(\dfrac{\leb{I_{n,i}}}{2^n}\right)^\rho \mathds{1}_{\mathcal{E}_{n,i}}\right) \\ \leq 
	& 2^{\rho n\left(\frac{\gamma}{H}-1\right)} \sum_{0}^{\lfloor 2^{n-1-n\frac{\gamma}{H}}\rfloor} \mathbb{P}\left(\mathcal{E}_{n,i}\right)  \\ \leq
	& c 2^{n\frac{H-\gamma}{H}(1-H+\eta -\rho)}.
\end{align*}
Thus, the Fubini Theorem entails $\mathbb{E}\left[\sum_{n=1}^{\infty}\nu_{\rho}^{n}(\eg) \right]< + \infty $ as soon as $\rho > 1 - H + \eta$. This implies that for such $\rho$’s, the sum $\sum_{n=1}^{\infty}\nu_{\rho}^{n}(\eg)$ is finite almost surely. In particular, $\dhaus{\eg}\leq \rho$, for every $\rho > 1 - H + \eta$.
Since such a relation holds for an arbitrary (small) $\rho > 0$, we deduce~\eqref{eq:haus_upper} as desired.

\bibliographystyle{plain}
\bibliography{lib_pr}

\end{document}